\documentclass[11pt]{amsart}
\usepackage[margin= 2.0cm]{geometry}
\usepackage{amsmath}
\usepackage{amsfonts}
\usepackage{amssymb}
\usepackage{amsthm}
\usepackage{mathtools} 
\usepackage{latexsym}

\usepackage{enumerate}
\usepackage{cancel}
\usepackage{ragged2e}
\usepackage{blindtext}
\usepackage{cases}
\usepackage{empheq}
\usepackage{multicol}

\usepackage{wrapfig}
\usepackage{caption}

\usepackage{subfigure}
\usepackage{float}

\usepackage{booktabs}
\usepackage{array}

\usepackage{color}

\usepackage[color=white,linecolor=black]{todonotes}
   
\usepackage{mathrsfs}
\usepackage{fontenc} 
\usepackage{inputenc} 
\usepackage{enumitem}
\usepackage{verbatim}

\theoremstyle{plain}
\newtheorem{theorem}{Theorem}[section]
\newtheorem{proposition}[theorem]{Proposition}
\newtheorem{lemma}[theorem]{Lemma}
\newtheorem{corollary}[theorem]{Corollary}

\theoremstyle{definition}
\newtheorem{definition}[theorem]{Definition}

\theoremstyle{remark}
\newtheorem{remark}[theorem]{Remark}

\numberwithin{equation}{section} 
\numberwithin{figure}{section}   

\usepackage[square,comma,numbers,sort]{natbib}
\usepackage[colorlinks=true, pdfborder={ 0 0 0}]{hyperref}
\hypersetup{urlcolor=blue, citecolor=red}
\usepackage{url}
\makeatletter
\renewcommand{\section}{\@startsection{section}{1}{0pt}%
  {1.5ex plus .2ex minus .2ex}%
  {1.0ex plus .2ex}%
  {\normalfont\large\bfseries\raggedright}}
\renewcommand{\subsection}{\@startsection{subsection}{2}{0pt}%
  {1.25ex plus .2ex minus .2ex}%
  {1.0ex plus .2ex}%
  {\normalfont\normalsize\bfseries\raggedright}}
\makeatother
\renewenvironment{abstract}{%
  \begin{center}
    \textbf{\abstractname}
  \end{center}
}{}
\makeatother
\allowdisplaybreaks
\begin{document}
\makeatletter
\def\@settitle{\begin{center}%
  \normalfont\LARGE\bfseries \@title
  \end{center}%
}
\makeatother

\title[Physics-Informed Neural Network in Kuramoto-Sivashinsky equation]{Regularity and error estimates in physics-informed neural networks  for the Kuramoto-Sivashinsky equation}
\author[M. M. Rahman]{Mohammad Mahabubur Rahman}
\address[Mohammad Mahabubur Rahman]{School of Mathematical and Statistical Sciences, 
                Clemson University, 
       Clemson, SC, 29634, USA}
\email[Mohammad Mahabubur Rahman]{rahman6@clemson.edu}
\author[D. Verma]{Deepanshu Verma
}
\address[Deepanshu Verma
]{School of Mathematical and Statistical Sciences, 
                Clemson University, 
       Clemson, SC, 29634, USA}
\email[Deepanshu Verma
]{dverma@clemson.edu}
\maketitle{}
\begin{abstract}
Due to its nonlinearity, bi-harmonic dissipation, and backward heat-like term in the absence of a divergence-free condition, the $2$-D/$3$-D Kuramoto-Sivashinsky equation poses significant challenges for both mathematical analysis and numerical approximation. These difficulties motivate the development of methods that blend classical analysis with numerical approximation approaches embodied in the framework of the physics-informed neural networks (PINNs). In addition, despite the extensive use of PINN frameworks for various linear and nonlinear PDEs, no study had previously established rigorous error estimates for the Kuramoto-Sivashinsky equation within a PINN setting. In this work, we overcome the inherent challenges, and establish several global regularity criteria based on space-time integrability conditions in Besov spaces. We then derive the first rigorous error estimates for the PINNs approximation of the Kuramoto-Sivashinsky equation and validate our theoretical error bounds through numerical simulations.
\end{abstract}
\medskip

\section{Introduction}
 \subsection{Kuramoto-Sivashinsky Equation}
The Kuramoto-Sivashinsky equation  features a rich interplay between nonlinear advection, bi-harmonic higher-order  dissipation, and a  backward heat-type term, offering chaotic spatiotemporal dynamics. For instance, the Kuramoto-Sivashinsky equation has been widely used to study flame front instabilities and ion plasma instabilities. We recall the Kuramoto-Sivashinsky equation, which was originally proposed in the 1970s by Kuramoto and Tsuzuki in the studies of crystal growth \cite{34, 35}, as well as by Sivashinsky in the study of flame-front instabilities \cite{55} (see also \cite{42,56}).

Let us now introduce the $d$-dimensional Kuramoto-Sivashinsky equation in vector form, given by
\begin{equation}
\label{KSE}
\left\{
\begin{aligned}
&\partial_t u + (u \cdot \nabla) u + \lambda \Delta u + \Delta^2 u = 0 && \text{in } \mathbb{T}^d \times [0,T],\\
&u(0) = u_0 && \text{in } \mathbb{T}^d.
\end{aligned}
\right.
\end{equation}
Here, $u: \mathbb{T}^d \times[0,T] \to \mathbb{R}^d$, with $d \in\{2,3\}$, is the flame front velocity, $\lambda > 0$ is a constant, and $u_0:\mathbb{T}^d \to \mathbb{R}^d$ is the initial flame front velocity. We consider the Kuramoto-Sivashinsky  equation \eqref{KSE} on the $d$-dimensional torus $\mathbb{T}^d=\mathbb{R}^d / (2\pi \mathbb{Z}^d)$ with periodic boundary conditions. 

It is worth noting that the non-vanishing nonlinear term $(u \cdot \nabla)u$ in the $L^2$-energy estimate, due to the lack of a divergence-free condition and the absence of scaling-invariant solutions, makes the problem particularly challenging and intriguing in higher dimensions. This raises a fundamental question: Does the two-or three-dimensional Kuramoto-Sivashinsky equation admit a global strong solution for arbitrary initial data? This problem remains open. 
Motivated by this, it is natural to ask whether one can establish sufficient conditions in Besov spaces that address this question, at least partially, by ensuring the existence of a global strong solution, which we investigate in this manuscript.  

Having introduced the equation, we now review several results concerning global well-posedness (see, e.g, \cite{46, 59}) that have been obtained for the $1$-D case over the past decades; the equation has been shown to exhibit rich and interesting dynamics.  In contrast to the $1$-D case, the analysis of global well-posedness  for the $2$-D Kuramoto-Sivashinsky equation poses substantial mathematical challenges, one of the reasons being that the nonlinear term $(u \cdot \nabla u, u)$ does not vanish in the $L^2$-energy identity due to the lack of a divergence-free condition. Global well-posedness for sufficiently small initial data was first established in \cite{54} on the domain $[0, 2\pi] \times [0, 2\pi \epsilon]$ with $\epsilon > 0$ sufficiently small. Later works continued to develop such results-i.e., global existence for small initial data-and improved the sharpness of the smallness condition (see, e.g., \cite{1, 4, 33, 41} and references therein).  Further studies have investigated modified equations \cite{13, 17, 39,38, 50}. Faced with the challenge of exploring global well-posedness, the possibility of finite-time blow-up cannot be ruled out in the absence of additional regularity conditions. To address this issue, one may resort to Prodi-Serrin-type regularity criteria, which provide sufficient conditions for preventing blow-up.  We refer readers to the literature \cite{37, CH, JU, FA} in this direction.

A function space is called critical for the  Kuramoto-Sivashinsky equation \eqref{KSE} if its norm is invariant under the scaling.  
In contrast to the Navier-Stokes equations, the  Kuramoto-Sivashinsky equation \eqref{KSE} generally lacks a  scaling-invariant solution. However, in the special case $\lambda = 0$, where the backward heat term $\lambda \Delta$ vanishes, the  Kuramoto-Sivashinsky equation \eqref{KSE} becomes invariant under the  following scaling
$$
u_\beta(t, x) := \beta^3 u(\beta^4 t, \beta x)\; \forall \beta>0,
$$
under which the Besov space $B^{-3}_{\infty,\infty}$ is critical. Establishing regularity criteria in this space typically requires the divergence-free condition to control the nonlinear terms and close the necessary estimates. 
 However, the absence of the divergence-free condition in the  Kuramoto-Sivashinsky equation makes it significantly difficult to derive regularity criteria in terms of this space. To illustrate these difficulties, consider taking the curl of the  Kuramoto-Sivashinsky equation \eqref{KSE}, which yields 
\begin{align}
\label{est 1.3k}
\partial_t \omega + \lambda \nabla \omega + \Delta^2 \omega = (\nabla \cdot u)\omega - (u \cdot \nabla)\omega + (\nabla \cdot \omega) u + (\omega \cdot \nabla) u,
\end{align}
where $\omega = \nabla \times u$. In particular, the terms $(\nabla \cdot u)\omega$ and $(\nabla \cdot \omega)u$ vanish under the incompressibility condition $\nabla \cdot u = 0$. The last term in \eqref{est 1.3k}, known as the vortex stretching term, $(\omega \cdot \nabla)u$ vanishes in the two-dimensional case, whereas it persists in three dimensions within the framework of the Navier-Stokes equations. In contrast, none of the terms on the right-hand side of \eqref{est 1.3k} vanish for the Kuramoto-Sivashinsky equation. The absence of a divergence-free condition in the Kuramoto-Sivashinsky equation leads to the
failure of standard structural cancellations, thereby complicating energy estimates-even in the $2$-D case.

Despite these challenges,  we aim to establish such criteria; specifically; we first prove  $u \in H^k(\mathbb{T}^d \times [0,T])$ for some $k \in \mathbb{N}$, and then derive global regularity criteria in terms of $\|u\|_{L^p(0,T; \dot{B}^{s}_{q,r})}$ and $\|u_i\|_{L^p(0,T; \dot{B}^{s}_{q,r})}$.

With the above review of recent progress on the Kuramoto-Sivashinsky equation and the classical results in mind, we now turn to recent developments in physics-informed neural networks (PINNs) and outline how these advances enable us to construct a PINN framework for the Kuramoto-Sivashinsky equation-an approach which, to the best of our knowledge, has not been previously established and thus represents a unique and substantive contribution to the study of the Kuramoto-Sivashinsky equation.

\subsection{Physics-Informed Neural Networks}
Physics-informed neural networks (PINNs) have emerged as an efficient method for solving partial differential equations (PDEs), particularly in complex or nonlinear settings. While the foundational ideas date back to the 1990s (see, \cite{15a, 40l, 41p}), there has been a notable resurgence of interest in recent years, particularly in the works \cite{RAM, RA}. These approaches incorporate the underlying physical laws directly into the training process by penalizing the residuals of the governing PDEs, along with those arising from the initial and boundary conditions.  For a comprehensive overview of recent developments, we refer the reader to \cite{14p, KA21} and references therein. In addition, deriving error estimates for PINNs has recently drawn considerable attention from both mathematicians and physicists, leading to extensive investigations. For the reader's reference, we list a series of works \cite{DE, DE1, MS, MS1} that explore error analyses in this direction. Establishing rigorous error bounds for PINNs is particularly delicate for equations such as the two-dimensional Kuramoto-Sivashinsky equation.  Although the implementation of the PINN framework for the Kuramoto-Sivashinsky equation \eqref{KSE} is nontrivial, the  Kuramoto-Sivashinsky equation nevertheless serves as an excellent test case, as it can be regarded as a backward-type heat equation when the nonlinear and biharmonic terms are dropped-an inverse-problem setting in which PINNs perform remarkably well for both forward and inverse cases, whereas classical numerical schemes such as finite difference and finite element methods are primarily effective for forward computations. We highlight here at least two fundamental difficulties associated with the  Kuramoto-Sivashinsky equation \eqref{KSE}:\begin{enumerate}
    \item \textbf{(Difficulty 1)} One of the fundamental difficulties arises from the presence of both a biharmonic term and an anti-diffusive term in the  Kuramoto-Sivashinsky equation \eqref{KSE} which induces instability, and the application of Green's formula introduces additional higher-order boundary contributions originating from the biharmonic operator, whose control requires the use of trace inequalities thereby requiring that the solution possess sufficient smoothness on the boundary $\partial D$, typically at least of class $\mathbb{C}^4$, where $D$ is a periodic domain. 

    \item \textbf{(Difficulty 2)} The nonlinear term~$(u \cdot \nabla)u$ does not exhibit cancellation at the $L^2$-level of the energy or error analysis, thereby complicating both the derivation of rigorous error estimates and the corresponding computational implementation.
\end{enumerate}
 However, to gain a better understanding of how the PINNs framework for the  Kuramoto-Sivashinsky equation \eqref{KSE} operates, we refer the reader to the detailed discussion in Section~\ref{subsec:pinns}.

Although extensive research has been conducted on PINNs in recent years, applying them to the Kuramoto-Sivashinsky equation appears to be unique, as this equation-unlike the Navier-Stokes or Burgers equations-contains an anti-diffusive term, a higher-order biharmonic term, and lacks the divergence-free condition. Despite these challenges, we successfully overcome these difficulties by employing the PINN framework for the Kuramoto-Sivashinsky equation, establishing rigorous error estimates and numerical approximations, and visualizing the results by comparing the PINN solution with the exact solution, including graphical representations of the absolute error.
  
\subsection{Plan for the Paper:}
This paper is organized as follows. In Section \ref{sec.2}, we present the necessary preliminaries along with several auxiliary results established in this manuscript, such as Lemma \ref{12K} and Corollary \ref{C3.2}. In Section \ref{sec.3}, we  prove the main theoretical results, including the regularity results (Theorems and Propositions \ref{Th 3.2}-\ref{Theorem 3.10kM}) as well as the upper bounds on the PINN residuals and the corresponding error estimates (Theorems \ref{3.11K}, \ref{Th 4.3}, and \ref{4.4}). In Section \ref{sec4}, we present  the numerical experiments. Finally, in Appendix~\ref{Appen}, we present the local existence and regularity result, along with several classical results that are used frequently throughout this manuscript.

\section{Preliminaries}\label{sec.2}
We now recall and set up the notations and function spaces that will be frequently used throughout the manuscript
\subsection{Notations and Spaces}
To define Homogeneous  and inhomogeneous Sobolev spaces, we recall the Fourier series representation of $f$ on the $d=\{2,3\}$-dimensional torus: $\mathbb{T}^d$
$$
f(x) = \sum_{k \in \mathbb{Z}^d} \hat{f}(k) e^{ik \cdot x}, 
$$
and the inhomogeneous and homogeneous Sobolev norms
$$
\lVert  f \rVert_{H^s} =\left( \sum_{k \in \mathbb{Z}^d} (1 + \lvert k\rvert^{2s})\lvert \hat{f}(k)\rvert^2 \right)^{\frac12}, 
\quad 
\lVert f \rVert_{\dot{H}^s} = \left( \sum_{k \in \mathbb{Z}^d} \lvert k\rvert^{2s}\lvert \hat{f}(k)\rvert^2 \right)^{\frac12},
$$
respectively. 

The following section provides the definition of homogeneous Besov spaces and related foundational facts, as described in \cite{THS}.
\subsection{Besov Spaces}
We define the $d$-dimensional $2\pi$-periodic torus as
$
\mathbb{T}^d = \mathbb{R}^d / (2\pi \mathbb{Z}^d).$ We next  define the spaces $\mathcal{D}(\mathbb{T}^d)$, $\mathcal{D}_0(\mathbb{T}^d)$ and $\mathcal{S}(\mathbb{Z}^d)$ by

$$
\mathcal{D}(\mathbb{T}^d) \equiv \left\{ f \in C^\infty(\mathbb{T}^d) ; f \ \text{is } 2\pi\text{-periodic on each component } x_1, \ldots, x_d \right\},
$$
$$
\mathcal{D}_0(\mathbb{T}^d) \equiv \left\{ f \in \mathcal{D}(\mathbb{T}^d) ; \int_{[-\pi, \pi]^d} f(x) \, dx = 0 \right\},
$$
$$
\mathcal{S}(\mathbb{Z}^d) \equiv \left\{ g : \mathbb{Z}^d \to \mathbb{R}^d ; \forall s \ge 0, \ \exists c = c(g,s) > 0 \ \text{s.t.} \ \sup_{m \in \mathbb{Z}^d} (1 + |m|^2)^{\frac{s}{2}} |g(m)| < c \right\},
$$
and let $\mathcal{D}'(\mathbb{T}^d)$, $\mathcal{D}'_0(\mathbb{T}^d)$ and $\mathcal{S}'(\mathbb{Z}^d)$ be dual spaces of $\mathcal{D}(\mathbb{T}^d)$, $\mathcal{D}_0(\mathbb{T}^d)$ and $\mathcal{S}(\mathbb{Z}^d)$, respectively. In addition, we define the toroidal Fourier transform (the Fourier series) $\mathcal{F}^d: \mathcal{D}(\mathbb{T}^d) \to \mathcal{S}(\mathbb{Z}^d)$ by
\begin{align*}
\mathcal{F}_{\mathbb{T}^d} f(m) \equiv \frac{1}{(2\pi)^d} \int_{[-\pi,\pi]^d} f(x) e^{-i m \cdot x} dx, \quad f \in \mathcal{D}(\mathbb{T}^d), \ m \in \mathbb{Z}^d,
\end{align*}
and the inversion $
\mathcal{F}^{-1}_{\mathbb{T}^d} : \mathcal{S}(\mathbb{Z}^d) \to C^\infty(\mathbb{T}^d)
$ by
$$
\mathcal{F}_{\mathbb{T}^d} \equiv \sum_{m \in \mathbb{Z}^d} g(m) e^{i m \cdot x}, \quad g \in \mathcal{S}(\mathbb{Z}^d), \ x \in \mathbb{T}^d.
$$
We can also define these transforms in dual spaces, $\mathcal{F}^d: \mathcal{D}'_0(\mathbb{T}^d) \to \mathcal{S}'(\mathbb{Z}^d)$ and $\mathcal{F}^{-1, d}: \mathcal{S}'(\mathbb{Z}^d) \to \mathcal{D}'(\mathbb{T}^d)$, by
$$
\langle \mathcal{F}_{\mathbb{T}^d}, \varphi \rangle \equiv \langle f, \mathcal{F}^{-1, d} \varphi(-\cdot) \rangle, \quad f \in \mathcal{D}'_0(\mathbb{T}^d), \ \varphi \in \mathcal{S}(\mathbb{Z}^d),
$$
$$
\langle \mathcal{F}_{\mathbb{T}^d}, \psi \rangle \equiv \langle g, \mathcal{F}^d \psi(-\cdot) \rangle, \quad g \in \mathcal{S}'(\mathbb{Z}^d), \ \psi \in \mathcal{D}(\mathbb{T}^d).
$$
Moreover, we define the convolution of $f, g \in \mathcal{D}(\mathbb{T}^d)$ by
$$
(f * g)(x) \equiv \int_{[-\pi,\pi]^d} f(x-y) g(y) \, dy, \quad x \in \mathbb{T}^d.
$$
We also define the convolution of $(h, f) \in \mathcal{D}(\mathbb{T}^d) \times \mathcal{D}'(\mathbb{T}^d)$ by
$$
\langle h * f, \varphi \rangle \equiv \langle f, h(-\cdot) * \varphi \rangle = \int_{[-\pi,\pi]^d} \langle f, h(y - \cdot) \rangle \varphi(y) dy, \quad \varphi \in \mathcal{D}(\mathbb{T}^d),
$$
and it is seen that $h * f$ is actually in $\mathcal{D}(\mathbb{T}^d)$. We can also define that of $(h,f) \in \mathcal{D}_0(\mathbb{T}^d) \times \mathcal{D}'_0(\mathbb{T}^d)$ by a similar way.

Now let us define some important operators and spaces related to $\mathcal{D}_0(\mathbb{T}^d)$ and $\mathcal{D}'_0(\mathbb{T}^d)$. Since $\mathcal{F}^d f(0) = 0$ for every $f \in \mathcal{D}_0(\mathbb{T}^d)$, we can define the toroidal Riesz potential $I^s$ with $s \in \mathbb{R}$ and Riesz transform $R_k$ with $k = 1, \ldots, d$, on $\mathcal{D}'_0(\mathbb{T}^d)$ by
$$
I^s \equiv \mathcal{F}_{\mathbb{T}^d}^{-1} \left[ | m|^{-s} \mathcal{F}_{\mathbb{T}^d} f(m) \right], \quad R_k f \equiv \mathcal{F}_{\mathbb{T}^d}^{-1} \left[ i m_k |m|^{-1} \mathcal{F}_{\mathbb{T}^d} f(m) \right].
$$
Secondly, we define the homogeneous toroidal Besov spaces.
Take a non-negative smooth function $\phi \in C^\infty(\mathbb{R}^d)$ such that
\begin{align}
0 \le \phi \le 1, \quad \operatorname{supp} \phi \subset \{\, x \in \mathbb{R}^d \; ; \; 1/2 < |\xi| < 2 \,\}, \quad \sum_{j=-\infty}^{\infty} \phi( 2^{-j} \xi) = 1, \quad \forall \xi \in \mathbb{R}^d \setminus \{0\}.
\label{1A}
\end{align}
Then set
\begin{align}
\phi_j(\xi) \equiv \phi(2^{-j} \xi), \quad 
\varphi_j \equiv \mathcal{F}^{-1}_{\mathbb{T}^d} \big( \phi_j(\,\cdot\,) \big), \quad j \in \mathbb{Z}.
\label{2B}
\end{align}
Each $\varphi_j$ belongs to $\mathcal{D}_0(\mathbb{T}^d)$.
Moreover, since $\mathcal{F}_{\mathbb{T}^d} f(0) = 0$ for every $f \in \mathcal{D}_0(\mathbb{T}^d)$, it follows from \eqref{1A} and \eqref{2B} that
$$
\sum_{j=-\infty}^{\infty} \varphi_j * f = f, \quad \forall f \in \mathcal{D}'_0(\mathbb{T}^d).
$$
According to the above family $\{\varphi_j\}_{j\in\mathbb{Z}}$, we define the homogeneous toroidal Besov space $\dot{B}^s_{p,q}(\mathbb{T}^d)$ for $s \in \mathbb{R}$, $1 \le p, q \le \infty$ by
$$
\dot{B}^s_{p,q}(\mathbb{T}^d) \equiv
\left\{ f \in \mathcal{D}'_0(\mathbb{T}^d) \; ; \; \| f \|_{\dot{B}^s_{p,q}(\mathbb{T}^d)} < \infty \right\},
$$
with the norm
$$
\|f\|_{\dot{B}^s_{p,q}(\mathbb{T}^d)} \equiv
\begin{cases}
\left( \sum_{j=-\infty}^{\infty} \left( 2^{sj} \, \| \varphi_j * f \|_{L^p(\mathbb{T}^d)} \right)^q \right)^{1/q}, & 1 \le q < \infty, \\
\sup_{j\in\mathbb{Z}} \ 2^{sj} \, \| \varphi_j * f \|_{L^p(\mathbb{T}^d)}, & q = \infty.
\end{cases}
$$
Here $L^p(\mathbb{T}^d)$ is the space of $2\pi$-periodic measurable functions with the norm
$$
\|f\|_{L^p(\mathbb{T}^d)} \equiv
\begin{cases}
\left( \int_{[-\pi,\pi]^d} |f(x)|^p dx \right)^{1/p}, & 1 \le p < \infty, \\
\operatorname*{ess\,sup}_{x\in\mathbb{T}^d} |f(x)|, & p = \infty.
\end{cases}
$$
Revisiting the notations and function spaces introduced above, we next define the concepts of a strong solution to the Kuramoto-Sivashinsky equation \eqref{KSE} and establish a corresponding Lemma \ref{12K} ensuring that each term in the  Kuramoto-Sivashinsky equation \eqref{KSE} is well-defined  in~$L^2$.

\subsection{The Concepts of Solutions }  
\begin{definition}
We call $u \in L^\infty([0, T); H^2(\mathbb{T}^d))\cap L^2([0, T); H^4(\mathbb{T}^d))$ is a strong solution to the  Kuramoto-Sivashinsky equation \eqref{KSE} over a time interval \( [0, T) \) if for any \( \psi \in C^\infty(\mathbb{T}^d) \), 
\begin{align}
\langle \partial_t u, \psi \rangle + (u \cdot \nabla u, \psi) + (\Delta u, \Delta \psi) = \lambda (\nabla u, \nabla \psi). 
\end{align}  
\end{definition}
\begin{lemma}
\label{12K}Let $u$ be a strong solution, then $\partial_t u, (u \cdot \nabla)u, \Delta u, \textrm{and}\; \Delta^2 u \in L^2([0, T); L^2(\mathbb{T}^d))$.
\end{lemma}

\begin{proof} From the assumption $ u \in L^2([0, T); H^4)$ it follows immediately that $\Delta u \in L^2([0, T); L^2)$ and $\Delta^2 u \in L^2([0, T); L^2)$. For the nonlinear term, note that \( H^2 \hookrightarrow L^\infty \), so
\begin{align*}
\int_0^T \lVert (u\cdot \nabla) u\rVert_{L^2}^{2} dt \leq \int_0^T \lVert u(t)\rVert_{L^\infty}^2 \lVert \nabla u(x, t)\rVert_{L^2}^2  \, dt \leq C \|u\|_{L^\infty(0, T; H^1)} \|u\|_{L^2(0, T; H^2)}.
\end{align*}
Let  \( \phi \in C^\infty(\mathbb{T}^d) \) and  for every \( t \in [0, T] \), we have
\begin{align*}
\int_{\mathbb{T}^d} u(t) \cdot \phi \, dx &= \int_{\mathbb{T}^d} u^0 \cdot \phi \, dx + \int_0^t \int_{\mathbb{T}^d} \nabla u : \nabla \phi-\Delta u \cdot \Delta \phi- (u \cdot \nabla)u \cdot \phi \, dx \, ds
\nonumber\\
&= \int_{\mathbb{T}^d} u^0 \cdot \phi \, dx - \int_0^t \int_{\mathbb{T}^d}  \Delta^2 u\cdot \phi+ \Delta u \cdot \phi + (u \cdot \nabla)u \cdot \phi \, dx \, ds.
\end{align*}
Proposition \ref{Prop 2.6} allows us to take the time derivative of this equality to obtain, for every \( t \in [0, T) \),
\begin{align*}
\int_{\mathbb{T}^d} \partial_t u \cdot \phi \, dx =- \int_{\mathbb{T}^d}  \Delta^2 u\cdot \phi+ \Delta u \cdot \phi + (u \cdot \nabla)u \cdot \phi \, dx.
\end{align*}
and since $\phi$ was arbitrary, it follows that $\partial_t u= -(\Delta^2 u+ \Delta u  + (u \cdot \nabla)u)$. By the previous estimates, this is indeed in $L^2([0, T); L^2)$.
\end{proof}
The local existence and regularity  in higher Sobolev spaces, Theorem \ref{2.1H} is standard; however, since we  are unable to locate a precise reference in the existing literature for the  higher Sobolev spaces, so we provide a sketch of the proof in the Appendix for the reader's convenience.

However, depending upon Theorem \ref{2.1H}, we establish the following Corollary \ref{C3.2}  that proves  that $u$ is Sobolev regular, i.e., $u \in H^k(\mathbb{T}^d \times [0,T])$ for some $k \in \mathbb{N}$. 
\begin{corollary}
\label{C3.2}
If $k \in \mathbb{N}$ and $u_0 \in H^m(\mathbb{T}^d)$ with $m > \frac{d}{2} + 4k$, then
there exist $T > 0$ and a classical solution $u$ to the  Kuramoto-Sivashinsky equation \eqref{KSE} such that
$u \in H^k(\mathbb{T}^d \times [0,T])$ and $u(t=0) = u_0$.
\end{corollary}
\begin{proof}
 The proof follows directly from Theorem \ref{2.1H} for $k = 1$. Therefore, let
$k \geq 2$ be arbitrary and assume that $m > \frac{d}{2} + 4k$ and $u_0 \in H^m(\mathbb{T}^d)$.
By Theorem \ref{2.1H}, there exists $T > 0$ and a classical solution $u$ to the  Kuramoto-Sivashinsky equation \eqref{KSE} such that $u(t=0) = u_0$ and $u \in C([0,T]; H^m(\mathbb{T}^d)) \cap C^1([0,T]; H^{m-4}(\mathbb{T}^d))$.
Taking the time derivative of the  Kuramoto-Sivashinsky equation \eqref{KSE}, we find that $u_{tt} \in C([0,T]; H^{m-8}(\mathbb{T}^d))$
and therefore $u \in C^2([0,T]; H^{m-8}(\mathbb{T}^d))$. Repeating these steps, one can prove that
$u \in \cap_{\ell = 0}^k C^{\ell}([0,T]; H^{m-4\ell}(\mathbb{T}^d))$,
and therefore the conclusion follows from this observation since $ \ell+ m - 4\ell \geq k$ for all $0 \leq \ell \leq k$ if $m > \frac{d}{2} + 4k$.
\end{proof}
\begin{remark}
Corollary \ref{C3.2} is indeed a special feature of the KSE, arising from the fourth-order biharmonic diffusion.

Several key results, including Corollary \ref{C3.2}, used in deriving the upper bound of the PINN residual for the approximation, motivate us to revisit related preliminaries on neural networks introduced in \cite{DE}.

\end{remark}
\subsection{Neural Networks}
\label{4.1N}
\begin{definition}
Let $R\in(0,\infty]$, $L,W\in\mathbb{N}$ and $l_0,\ldots,l_L\in\mathbb{N}$. 
Let $\sigma:\mathbb{R}\to\mathbb{R}$ be a $C^4$ activation function. Define
\[
\Theta=\Theta_{L,W,R}:=
\bigcup_{\substack{L'\in\mathbb{N}\\ L'\le L}}
\ \bigcup_{\,l_0,\ldots,l_{L'}\in\{1,\ldots,W\}}
\ \prod_{j=1}^{L'}
\Big( [-R,R]^{\,l_j\times l_{j-1}} \times [-R,R]^{\,l_j}\Big).
\]
For $\theta\in\Theta_{L,W,R}$, we write $\theta_j:=(W_j,b_j)$ and define the affine maps 
\[
A_j^\theta:\mathbb{R}^{l_{j-1}}\to\mathbb{R}^{l_j}, 
\qquad 
A_j^\theta(x)=W_j x + b_j, 
\quad 1\le j\le L,
\]
and the layer functions
\[
f_j^\theta(z)=
\begin{cases}
A_j^\theta(z), & j=L,\\[2mm]
(\sigma\circ A_j^\theta)(z), & 1\le j<L.
\end{cases}
\]
The realization of the neural network associated with parameter $\theta$ is the function $u_\theta:\mathbb{R}^{l_0}\to\mathbb{R}^{l_{L}}$ given by
\[
u_\theta(z)=\big(f_{L}^\theta\circ f_{L-1}^\theta\circ\cdots\circ f_1^\theta\big)(z),
\qquad z\in\mathbb{R}^{l_0}.
\]
In the setting of approximating the Kuramoto-Sivashinsky equation \eqref{KSE}, we set $l_0=d+1$ and $z=(x,t)$. 
We call $u_\theta$ the realization of the neural network with $L$ layers and widths $(l_0,\ldots,l_{L})$. The first $L-1$ layers are called \emph{hidden layers}. 
For $1\le j\le L$, layer $j$ has width $l_j$, and $W_j$ and $b_j$ are the weights and biases of layer $j$. The \emph{width} of $u_\theta$ is defined as $\max(l_0,\ldots,l_{L})$. 
If $L=2$, we say $u_\theta$ is a \emph{shallow} neural network, while if $L\ge 3$, we say $u_\theta$ is a \emph{deep} neural network.
\end{definition}

\subsection{Quadrature rules}\label{suq}
We recall some basic results on numerical quadrature rules to approximate the integrals of functions. Let $\Omega \subset \mathbb{R}^d$ with $d=\{2,3\}$ and $g \in L^1(\Omega)$. To approximate
$$
\int_\Omega g(y)\,dy,
$$
a quadrature rule selects points $y_m \in \Lambda$ and positive weights $w_m$ for $1 \le m \le M$, and forms
$$
\frac{1}{M}\sum_{m=1}^M w_m\, g(y_m) \;\approx\; \int_\Omega g(y)\,dy.
$$
The accuracy depends on the quadrature rule, the number of  quadrature points $M$, and the regularity of $g$. Since in our setting ($d \leq 3$), we use the standard deterministic   numerical quadrature points.

Partition $\Omega$ into $M \sim N^d$ cubes of edge length $1/N$, and let $\{y_m\}_{m=1}^M$ be their midpoints. The midpoint rule is defined by
$$
Q_M^\Omega[g] \;:=\; \frac{1}{M} \sum_{m=1}^M g(y_m),
$$
and satisfies the error estimate
\begin{align}
\Bigl|\;\int_\Omega g(y)\,dy - Q_M^\Omega[g]\;\Bigr| \;\le\; C_g\, M^{-2/d},
\qquad C_g \lesssim \|g\|_{C^2}.
\label{quad1}
\end{align}
\subsection{Physics-informed neural networks}\label{subsec:pinns}
We seek deep neural networks $u_\theta : D \times [0,T] \to \mathbb{R}^d$, parameterized by $\theta \in \Theta$, constituting the weights and biases, that approximate the solution $u$ of the  Kuramoto-Sivashinsky equation \eqref{KSE}. To this end, the key idea behind PINNs is to consider pointwise residuals, defined in the setting of the Kuramoto-Sivashinsky equation \eqref{KSE} for any sufficiently smooth $v : D \times [0,T] \to \mathbb{R}^d$ as
\begin{equation}
\left\{
\begin{aligned}
\label{2.5PL}
\mathcal{R}_{\mathrm{PDE}}[v](x,t) &= \bigl(v_t + v \cdot \nabla v + \lambda \Delta v + \Delta^2 v \bigr)(x,t),\\
\mathcal{R}_{sb}^1[v](y,t) &= v(y,t) - \psi_1(y,t),\\
\mathcal{R}_{sb}^2[v](y,t) &= \Delta v(y,t) - \psi_2(y,t),\\
\mathcal{R}_{sb}^3[v](y,t) &= \Delta v(y,t) - \psi_3(y,t),\\
\mathcal{R}_{sb}^4[v](y,t) &= \Delta v(y,t) - \psi_4(y,t),\\
\mathcal{R}_t[v](x) &= v(x,0) - \varphi(x),
\end{aligned}
\right.
\end{equation}
for $x \in D$, $y \in \partial D$, $t \in [0,T]$. Here, $\psi_1, \psi_2, \psi_3, \psi_4 : \partial D \times [0,T] \to \mathbb{R}^d$ to specify the boundary data, and $\varphi : D \to \mathbb{R}^d$ is the initial condition.

Hence, within the PINN framework, one seeks $u_\theta$ such that all residuals are simultaneously minimized, by minimizing
\begin{align}
\label{2.6H}
E_G(\theta)^2
&=\int_{D\times[0,T]}\!\!\!\|\mathcal{R}_{\mathrm{PDE}}[u_\theta](x,t)\|_{\mathbb{R}^d}^2\,dx\,dt
+\int_{\partial D\times[0,T]}\!\!\!\bigl(
\|\mathcal{R}_{sb}^1[u_\theta](x,t)\|_{\mathbb{R}^d}^2
+\|\mathcal{R}_{sb}^2[u_\theta](x,t)\|_{\mathbb{R}^d}^2
\bigr)\,ds(x)\,dt
\nonumber\\
&\quad
+\int_{\partial D\times[0,T]}\!\!\!\bigl(
\|\mathcal{R}_{sb}^3[u_\theta](x,t)\|_{\mathbb{R}^d}^2
+\|\mathcal{R}_{sb}^4[u_\theta](x,t)\|_{\mathbb{R}^d}^2
\bigr)\,ds(x)\,dt+\int_{D}\|\mathcal{R}_t[u_\theta](x)\|_{\mathbb{R}^d}^2\,dx.
\end{align}
For simplicity, all weights are set to one. The quantity $E_G(\theta)$, often referred to as the generalization error of the neural network $u_\theta$, involves integrals that cannot be evaluated exactly in practice. Instead, the integrals in \eqref{2.6H} are approximated using numerical quadrature, as introduced in Section \ref{suq}. 

Accordingly, the (squared) training loss for PINNs is defined as
\begin{align}
\label{4.2ST}
E_T(\theta,S)^2 
&= E^{\mathrm{PDE}}_T(\theta,S_{\mathrm{int}})^2 
+ E^{sb,1}_T(\theta,S_{sb,1})^2 
+ E^{sb,2}_T(\theta,S_{sb,2})^2 
+ E^{sb,3}_T(\theta,S_{sb,3})^2 
+ E^{sb,4}_T(\theta,S_{sb,4})^2 
\nonumber\\
&
+ E^{t}_T(\theta,S_t)^2 
\nonumber\\
&=\sum_{n=1}^{N_{\mathrm{int}}} w^{\mathrm{int}}_n 
   \|\mathcal{R}_{\mathrm{PDE}}[u_\theta](t^{\mathrm{int}}_n, x^{\mathrm{int}}_n)\|_{\mathbb{R}^d}^2
 +\sum_{n=1}^{N_{sb}} w^{sb}_n 
   \big(\|\mathcal{R}_{sb}^1[u_\theta](t^{sb}_n, x^{sb}_n)\|_{\mathbb{R}^d}^2
   +\|\mathcal{R}_{sb}^2[u_\theta](t^{sb}_n, x^{sb}_n)\|_{\mathbb{R}^d}^2
\nonumber\\
&+\|\mathcal{R}_{sb}^3[u_\theta](t^{sb}_n, x^{sb}_n)\|_{\mathbb{R}^d}^2
+\|\mathcal{R}_{sb}^4[u_\theta](t^{sb}_n, x^{sb}_n)\|_{\mathbb{R}^d}^2\big) + \sum_{n=1}^{N_t} w^t_n\|\mathcal{R}_t[u_\theta](x^t_n)\|_{\mathbb{R}^d}^2.
\end{align}
Here, $S=(S_{\mathrm{int}},S_{sb,1},S_{sb,2}, S_{sb,3},S_{sb,4},S_t)$ denotes the set of quadrature points in the corresponding domains (respectively $D\times[0,T]$, $\partial D\times[0,T]$, and $D$), and $w_n^\ast$ are the associated quadrature weights. 

A trained PINN $u^* = u_{\theta^*}$ is then defined as a (local) minimizer of
\begin{equation}
\theta^*(S) = \arg\min_{\theta\in\Theta} E_T(\theta,S)^2,
\end{equation}
found by stochastic optimization methods such as ADAM or L-BFGS. Using the midpoint quadrature rule $Q_M$, the compact form of the training loss reads
\begin{align}
\label{34LK}
E_T(\theta,S)^2
&= Q^{\mathrm{int}}_M[R^2_{\mathrm{PDE}}]
 + Q^{sb,1}_M[R^2_{sb,1}]
 + Q^{sb,2}_M[R^2_{sb,2}]
 + Q^{sb,3}_M[R^2_{sb,3}]
 + Q^{sb,4}_M[R^2_{sb,4}]
 + Q^{t}_M[R^2_t].
\end{align}
In this manuscript, we address the following key theoretical questions for the Kuramoto-Sivashinsky equation:
\begin{itemize}
\item  \textbf{[Q1.]} \label{Q1.} Given a tolerance $\varepsilon > 0$, do there exist neural networks, such that the corresponding generalization error, $E_G(\overline{\theta})$ in \eqref{2.6H} and training error, $E_T(\hat{\theta}, S)$ in \eqref{4.2ST} are small, i.e.,
\[
E_G(\overline{\theta}),\; E_T(\hat{\theta}, S) < \varepsilon ?
\]

\item \textbf{[Q2.]} Given a PINN $u^*$ with small generalization error, is the corresponding total error $\lVert u - u^* \rVert$ small, i.e.,
\[
\lVert u - u^* \rVert < \delta(\varepsilon),
\]
for some $\delta(\varepsilon) \sim O(\varepsilon)$, with respect to a suitable norm $\lVert \cdot \rVert$, and with $u$ denoting the solution of the  Kuramoto-Sivashinsky equation \eqref{KSE}?

\item \textbf{[Q3.]} Given a small training error $E_T(\theta^*)$ and a sufficiently large training set $S$, is the corresponding generalization error $E_G(\theta^*)$ also proportionately small?
\end{itemize}

We now turn our focus to the following section, where despite the difficulties, we successfully establish integrability conditions that guarantee the regularity of solutions and yield rigorous error estimates.
\section{Main Results} 
The main results of this manuscript consist of Prodi-Serrin-type global regularity criteria presented in Section \ref{3.1KH}, and several $L^2$-error estimates for the PINNs in Section \ref{3.2}, expressed in terms of the residuals defined above in \ref{2.5PL}.

\label{sec.3}
\subsection{The Regularity Results}
\label{3.1KH}
We next prove the sufficient space-time integrability conditions for global regularity, as stated in Theorems \ref{Th 3.2} to \ref{Theorem 3.10kM}.
\begin{theorem}
\label{Th 3.2}
Let $u_0 \in H^{2}({\mathbb{T}^d})$. If $ u \in  
L^{\frac{8}{5}}(0,T,\dot{B}_{\infty, \infty}^{-\frac{1}{2}}(\mathbb{T}^d))$,  
then the  Kuramoto-Sivashinsky equation \eqref{KSE} admits a global strong solution, i.e., $u \in L^\infty([0, T]; H^2(\mathbb{T}^d))\cap L^2([0, T]; H^4(\mathbb{T}^d))$.
\end{theorem}  
\begin{proof}
Taking $L^2 (\mathbb{T}^d)$-inner products on the  Kuramoto-Sivashinsky equation \eqref{KSE}  with $u$
    \begin{align}
    \frac{1}{2} \frac{d}{dt} \lVert u \rVert_{L^{2}}^{2}+  \lVert \Delta u \rVert_{L^{2}}^{2} &=  - \int_{\mathbb{T}^2} (u\cdot\nabla) u \cdot u dx+ \lambda \lVert \nabla u \rVert_{L^{2}}^{2}
    \nonumber\\
    & \leq c \lVert  u\rVert_{L^4}^{2} \lVert \nabla u\rVert_{L^2} + \frac{1}{4}\lVert \Delta u\rVert_{L^2}^{2}+c\lVert  u\rVert_{L^2}^2
    \nonumber\\
    & \leq c  \lVert  u\rVert_{\dot{B}_{\infty, \infty}^{-\frac{1}{2}}} \lVert \Lambda^{\frac{1}{2}}u\rVert_{L^2} \lVert \nabla u \rVert_{L^2} + \frac{1}{4}\lVert \Delta u\rVert_{L^2}^{2}+c\lVert  u\rVert_{L^2}^2
    \nonumber\\
    & \leq c \lVert  u\rVert_{\dot{B}_{\infty, \infty}^{-\frac{1}{2}}}^{\frac{8}{5}} \lVert  u\rVert_{L^2}^{2}+ \frac{1}{4}\lVert \Delta u \rVert_{L^2}^{2}+ \frac{1}{4}\lVert \Delta u\rVert_{L^2}^{2}+c\lVert  u\rVert_{L^2}^2,
\end{align} 
where we applied H\"older inequality, Lemma  \ref{lemma_3.12k}, inequality \ref{est 500k}, Young's inequality,  and Gr\"onwall's inequality to deduce $u \in L_t^{\infty} L_x^2 \cap L_t^{2}H_x^2$.

Taking $L^2 (\mathbb{T}^d)$-inner products on the  Kuramoto-Sivashinsky equation \eqref{KSE}  with $\Delta u$
\begin{align}
  \frac{1}{2} \frac{d}{dt} \lVert \nabla u \rVert_{L^2}^2+   \lVert \nabla \Delta u \rVert_{L^2}^2=&\lambda \lVert \Delta u \rVert_{L^2}^{2} - \int_{\mathbb{T}^2} (u\cdot\nabla) u \cdot \Delta u dx
  \nonumber\\
  & \leq   \lVert \nabla u\rVert_{L^2}  \lVert u\rVert_{L^{\infty}}  \lVert \Delta u\rVert_{L^2}+ \lambda \lVert \Delta u\rVert_{L^2}^2
    \nonumber\\
    & \leq \frac{1}{8}\lVert \nabla \Delta u \rVert_{L^2}^{2}+c  \lVert \nabla u\rVert_{L^2}^{2} \lVert \Delta u \rVert_{L^2}^{2 } + \frac{1}{8}\lVert \nabla \Delta u\rVert_{L^2}^{2}+c\lVert \nabla u\rVert_{L^2}^2,
    \end{align}
where we applied H\"older's inequality, inequality \eqref{est 500k}, $\dot{H}^2(\mathbb{T}^d) \hookrightarrow L^{\infty}(\mathbb{T}^d)$,  Young's,  Gagliardo-Nirenberg interpolation inequality, and we can complete the theorem with Gr\"onwall's inequality and the bootstrapping technique.
\end{proof}

\begin{proposition}\label{Proposition 2.3k}
Let $ u_1 \in L^2(0,T;\dot{B}_{\infty, \infty}^{-\frac{1}{2}}(\mathbb{T}^2)) \cap L^2(0,T; \dot{B}_{\infty, \infty}^{\frac{1}{2}}(\mathbb{T}^2))$, then $u_{2} \in L_{t}^{\infty} L_x^{2} \cap L_{t}^{2} H_{x}^{2}$.
\end{proposition}
\begin{proof}
Taking $L^{2}(\mathbb{T}^2)$-inner products with $u_{2}$ on the second component of the  Kuramoto-Sivashinsky equation \eqref{KSE} gives 
\begin{align}
\label{est 2.21kMS}
\frac{1}{2} \frac{d}{dt} \lVert u_{2} \rVert_{L^{2}}^{2}  + \lVert \Delta u_{2} \rVert_{L^{2}}^{2}& = - \int_{\mathbb{T}^2} (u\cdot\nabla) u_{2} u_{2} dx + \lambda \lVert \nabla u_{2} \rVert_{L^{2}}^{2}
\nonumber\\
&= - \int_{\mathbb{T}^2} u_{1} (\partial_{1} u_{2}) u_{2} dx - \int_{\mathbb{T}^2} \frac{1}{3} \partial_{2} (u_{2})^{3} dx+c\lVert u_{2} \rVert_{L^{2}}^{2}+ \frac{1}{4}\lVert \Delta u_{2} \rVert_{L^{2}}^{2}
\nonumber \\
&\leq \lVert u_1\rVert_{L^4} \lVert \nabla u_2\rVert_{L^4} \lVert u_2\rVert_{L^2}+c\lVert u_{2} \rVert_{L^{2}}^{2}+ \frac{1}{4}\lVert \Delta u_{2} \rVert_{L^{2}}^{2} 
\nonumber\\
\leq& c\lVert u_{1} \rVert_{\dot{B}_{\infty, \infty}^{-\frac{1}{2}}}^{\frac{1}{2}} \lVert u_{1} \rVert_{\dot{B}_{\infty, \infty}^{\frac{1}{2}}}^{\frac{1}{2}} \lVert u_{2} \rVert_{L^2} \lVert \Lambda^{\frac{3}{2}} u_{2} \rVert_{L^2}+ c \lVert u_{2} \rVert_{L^{2}}^{2}+ \frac{1}{4}\lVert \Delta u_{2} \rVert_{L^{2}}^{2} \nonumber\\
\leq&  \frac{1}{4} \lVert \Delta u_{2} \rVert_{L^{2}}^{2}+\frac{1}{4}\lVert \Delta u_{2} \rVert_{L^{2}}^{2} + c (\lVert u_1\rVert_{\dot{B}_{\infty, \infty}^{-\frac{1}{2}}}^{2} + \lVert u_{1}\rVert_{\dot{B}_{\infty, \infty}^{\frac{1}{2}}}^{2}+1) \lVert u_{2} \rVert_{L^{2}}^{2},
\end{align} 
where we applied H$\ddot{\mathrm{o}}$lder's inequality, Lemma \ref{lemma_3.12k}, inequality \eqref{est 500k}, and Young's inequality.
By applying the Gr\"onwall's inequality, \eqref{est 2.21kMS} implies the desired result. 
\end{proof}

\begin{proposition}\label{Proposition 2.4k}
 Let  $ u_1 \in L^2(0,T;\dot{B}_{\infty, \infty}^{-\frac{1}{2}}(\mathbb{T}^2)) \cap L^2(0,T; \dot{B}_{\infty, \infty}^{\frac{1}{2}}(\mathbb{T}^2))$, then $u_{1} \in L_{t}^{\infty} L_x^{2} \cap L_{t}^{2} H_{x}^{2}$. 
\end{proposition}

\begin{proof}
Taking $L^{2}(\mathbb{T}^2)$-inner products on the first component of the  Kuramoto-Sivashinsky equation $\eqref{KSE}$ with $u_{1}$ and compute  
\begin{align}\label{ est 2.2K}
\frac{1}{2} \frac{d}{dt} \lVert u_{1} \rVert_{L^{2}}^{2} + \lVert \Delta u_{1} \rVert_{L^{2}}^{2}& = - \int_{\mathbb{T}^2} (u\cdot\nabla) u_{1} u_{1} dx+ \lambda \lVert \nabla u_{1} \rVert_{L^{2}}^{2} 
\nonumber\\
 =&- \int_{\mathbb{T}^2} u_{1} \partial_{1} u_{1} u_{1} dx - \int_{\mathbb{T}^2} u_{2} \partial_{2} u_{1} u_{1}dx+\lambda \lVert \nabla u_{1} \rVert_{L^{2}}^{2} \nonumber\\
\leq& \lVert \nabla u_{1} \rVert_{L^{2}} \lVert u_{2} \rVert_{L^{\infty}}  \lVert u_{1} \rVert_{L^{2}}+\lambda \lVert \Delta u_{1} \rVert_{L^{2}} \lVert u_{1} \rVert_{L^{2}}  \nonumber\\
\leq& c\lVert u_{2} \rVert_{H^{2}} \lVert u_{1} \rVert_{L^{2}}^{\frac{3}{2}} \lVert \Delta u_{1} \rVert_{L^{2}}^{\frac{1}{2}}+ c\lVert u_{1} \rVert_{L^{2}}^{2}+ \frac{1}{4}\lVert \Delta u_{1} \rVert_{L^{2}}^{2} 
\nonumber\\
&\leq \frac{1}{2} \lVert \Delta u_{1} \rVert_{L^{2}}^{2} + c (\lVert u_{2} \rVert_{H^{2}}^{\frac{4}{3}}+1) \lVert u_{1} \rVert_{L^{2}}^{2}, 
\end{align} 
where we applied H$\ddot{\mathrm{o}}$lder's inequality, Sobolev embedding $H^{2}(\mathbb{T}^2) \hookrightarrow L^{\infty}(\mathbb{T}^2)$, inequality \eqref{est 500k},  and Young's inequality. This implies that $u_{1} \in L_{t}^{\infty} L_{x}^{2} \cap L_{t}^{2}H_{x}^{2}$ because we know $\lVert u_{2}(t) \rVert_{H^{2}} \in L_{t}^{2}$ due to Proposition \ref{Proposition 2.3k}.
\end{proof} 

\begin{theorem}
Let $u_0 \in H^2(\mathbb{T}^2)$. If $u_1 \in L^2(0,T;\dot{B}_{\infty, \infty}^{-\frac{1}{2}}(\mathbb{T}^2)) \cap L^2(0,T; \dot{B}_{\infty, \infty}^{\frac{1}{2}}(\mathbb{T}^2))$, then the  Kuramoto-Sivashinsky equation \eqref{KSE} admits a global strong solution, i.e., $u \in L^\infty([0, T]; H^2(\mathbb{T}^2))\cap L^2([0, T]; H^4(\mathbb{T}^2))$.
\end{theorem}
\begin{proof}
Taking $L^2 (\mathbb{T}^2)$-inner products on the  Kuramoto-Sivashinsky equation \eqref{KSE}  with $\Delta u$
\begin{align}
\frac{1}{2} \frac{d}{dt} \lVert& \nabla u \rVert_{L^2}^2+   \lVert \nabla \Delta u \rVert_{L^2}^2= - \int_{\mathbb{T}^2} (u\cdot\nabla) u \cdot \Delta u dx+\lambda \lVert \Delta u \rVert_{L^2}^{2}
\nonumber\\
&=\lambda \lVert \Delta u \rVert_{L^2}^{2} - \sum_{k=1}^{3} \int_{\mathbb{T}^2} \begin{pmatrix}
u_1 \partial_1 u_1 +u_2 \partial_2 u_1\\
u_1 \partial_1 u_2 +u_2 \partial_2 u_2
\end{pmatrix} \cdot \begin{pmatrix}
\partial_k^2 u_1\\
\partial_k^2 u_2
\end{pmatrix}
dx
\nonumber\\
&= \lambda \lVert \Delta u \rVert_{L^2}^{2} - \sum_{k=1}^{3} \int_{\mathbb{T}^2}
u_1 \partial_1 u_1 \partial_k^2 u_1+u_2 \partial_2 u_1 \partial_k^2 u_1+
u_1 \partial_1 u_2 \partial_k^2 u_2 +u_2 \partial_2 u_2 \partial_k^2 u_2   dx,
\end{align}
after applying  Proposition \ref{Proposition 2.3k} and Proposition \ref{Proposition 2.4k}, Gr\"onwall's inequality implies $u \in L^\infty([0, T]; H^1(\mathbb{T}^2))\cap L^2([0, T]; H^3(\mathbb{T}^2))$. This completes the proof of the Theorem by bootstrapping for the higher regularity.
\end{proof}

We now extend our analysis to the three-dimensional case and establish the results presented in \ref{prop_3.6k}-\ref{3.12H}. In particular, we derive a sufficient global regularity criteria involving both $u_1$ and $u_2$; nonetheless, it remains an open problem whether a similar criterion can be obtained based solely on one component, either $u_1$ or $u_2$.

\begin{proposition}\label {prop_3.6k}
Let $ u_1,u_2 \in L^2(0,T;\dot{B}_{\infty, \infty}^{-\frac{1}{2}}(\mathbb{T}^3)) \cap L^2(0,T; \dot{B}_{\infty, \infty}^{\frac{1}{2}}(\mathbb{T}^3))$, then $u_{3} \in L_{t}^{\infty} L_{x}^{2} \cap L_{t}^{2} H_{x}^{2}$.
\end{proposition}
\begin{proof}
Taking $L^{2}(\mathbb{T}^3)$-inner products with $u_{2}$ on the second component of the  Kuramoto-Sivashinsky equation \eqref{KSE} gives  
\begin{align}
\label{est 2.20k}
\frac{1}{2} \frac{d}{dt} \lVert u_{3} \rVert_{L^{2}}^{2} &+ \lVert \Delta u_{3} \rVert_{L^{2}}^{2}  = - \int_{\mathbb{T}^3} (u\cdot\nabla) u_{3} u_{3} dx+ \lambda \lVert \nabla u_{3} \rVert_{L^{2}}^{2}
\nonumber\\
&= - \int_{\mathbb{T}^3} u_{1} \partial_{1} u_{3} u_{3}dx- \int_{\mathbb{T}^3} u_{2} \partial_{2} u_{3} u_{3} dx+\lambda \lVert u_{3} \rVert_{L^{2}} \lVert \Delta u_{3} \rVert_{L^{2}} dx 
\nonumber \\
&\leq (\lVert u_1\rVert_{L^4}+ \lVert u_2\rVert_{L^4})  \lVert \nabla u_3\rVert_{L^4} \lVert u_3\rVert_{L^2}+\lambda \lVert u_{3} \rVert_{L^{2}} \lVert \Delta u_{3} \rVert_{L^{2}}
\nonumber\\
\leq& c(\lVert u_{1} \rVert_{\dot{B}_{\infty, \infty}^{\frac{1}{2}}}^{\frac{1}{2}} \lVert u_{1} \rVert_{\dot{B}_{\infty, \infty}^{-\frac{1}{2}}}^{\frac{1}{2}} +\lVert u_{2} \rVert_{\dot{B}_{\infty, \infty}^{\frac{1}{2}}}^{\frac{1}{2}} \lVert u_{2} \rVert_{\dot{B}_{\infty, \infty}^{-\frac{1}{2}}}^{\frac{1}{2}} )\lVert u_{3} \rVert_{L^2} \lVert \Lambda^{\frac{3}{2}} u_{3} \rVert_{L^2}+c \lVert u_{3} \rVert_{L^{2}}^{2}
\nonumber\\
&+ \frac{1}{4}\lVert \Delta u_{3} \rVert_{L^{2}}^{2} \nonumber\\
\leq & \frac{1}{4} \lVert \Delta  u_{3} \rVert_{L^{2}}^{2}+\frac{1}{4}\lVert \Delta u_{3} \rVert_{L^{2}}^{2}
\nonumber\\
&+ c (\lVert u_1\rVert_{\dot{B}_{\infty, \infty}^{\frac{1}{2}}}^{2} + \lVert u_1 \rVert_{\dot{B}_{\infty, \infty}^{-\frac{1}{2}}}^{2}+\lVert u_2\rVert_{\dot{B}_{\infty, \infty}^{\frac{1}{2}}}^{2} + \lVert u_2 \rVert_{\dot{B}_{\infty, \infty}^{-\frac{1}{2}}}^{2}+1) \lVert u_{3} \rVert_{L^{2}}^{2},
\end{align} 
where we applied H$\ddot{\mathrm{o}}$lder's inequality, inequality \eqref{est 500k}, Lemma \ref{lemma_3.12k}, and Young's inequality.
By applying the Gr\"onwall's inequality, \eqref{est 2.20k} implies the desired result. 
\end{proof} 
\begin{proposition} \label{Thm 3.28k}
Let $ u_1,u_2 \in L^2(0,T;\dot{B}_{\infty, \infty}^{-\frac{1}{2}}(\mathbb{T}^3)) \cap L^2(0,T;  L^2(0,T;\dot{B}_{\infty, \infty}^{\frac{1}{2}}(\mathbb{T}^3))$, then $u_{1},u_{2} \in L_{t}^{\infty} L_{x}^{2} \cap L_{t}^{2} H_{x}^{2}$.
\end{proposition}
\begin{proof}
Fix $k \in\{1,2\}$. Taking $L^{2}(\mathbb{T}^3)$-inner products on the k-th component of the  Kuramoto-Sivashinsky equation $\eqref{KSE}$ with $u_{k}$ and compute  as follows:
\begin{align}
\label{ est 2.2K}
&\frac{1}{2} \frac{d}{dt} \lVert u_{k} \rVert_{L^{2}}^{2} + \lVert \Delta u_{k} \rVert_{L^{2}}^{2} = - \int_{\mathbb{T}^3} (u\cdot\nabla) u_{k} u_{k} + \lambda \lVert \nabla u_{k} \rVert_{L^{2}}^{2} 
\nonumber\\
 =&- \int_{\mathbb{T}^3} u_{1} \partial_{1} u_{k} u_{k}- \int_{\mathbb{T}^3} u_{2} \partial_{2} u_{k} u_{k} - \int_{\mathbb{T}^3} u_{3} \partial_{3} u_{k} u_{k}+\lambda \lVert \nabla u_{k} \rVert_{L^{2}}^{2} \nonumber \\
&\leq (\lVert u_1\rVert_{L^4}+ \lVert u_2\rVert_{L^4})  \lVert \nabla u_k\rVert_{L^4} \lVert u_k\rVert_{L^2}+\lVert u_3\rVert_{L^{\infty}} \lVert \nabla u_k\rVert_{L^2} \lVert u_k\rVert_{L^2}+\lambda \lVert u_{k} \rVert_{L^{2}} \lVert \Delta u_{k} \rVert_{L^{2}}
\nonumber\\
\leq& c(\lVert u_1\rVert_{\dot{B}_{\infty, \infty}^{-\frac{1}{2}}}^{\frac{1}{2}} \lVert u_1\rVert_{\dot{B}_{\infty, \infty}^{\frac{1}{2}}}^{\frac{1}{2}}+\lVert u_{2} \rVert_{\dot{B}_{\infty, \infty}^{-\frac{1}{2}}}^{\frac{1}{2}} \lVert u_{2} \rVert_{\dot{B}_{\infty, \infty}^{\frac{1}{2}}}^{\frac{1}{2}})\lVert u_{k} \rVert_{L^2} \lVert \Lambda^{\frac{3}{2}} u_{k} \rVert_{L^2}+ c\lVert u_{k} \rVert_{L^{2}}^{2}
\nonumber\\
&+ \frac{1}{4}\lVert \Delta u_{k} \rVert_{L^{2}}^{2}+c \lVert u_3\rVert_{H^2}^{2} \lVert u_k\rVert_{L^2}^{2}+\frac{1}{4} \lVert \nabla u_k\rVert_{L^2}^{2} \nonumber\\
\leq & \frac{1}{4} \lVert \Delta u_{k} \rVert_{L^{2}}^{2}+\frac{1}{4}\lVert \Delta u_{k} \rVert_{L^{2}}^{2} 
\nonumber\\
&+ c (\lVert u_1\rVert_{\dot{B}_{\infty, \infty}^{-\frac{1}{2}}}^{2} + \lVert u_1 \rVert_{\dot{B}_{\infty, \infty}^{\frac{1}{2}}}^{2}+\lVert u_2\rVert_{\dot{B}_{\infty, \infty}^{-\frac{1}{2}}}^{2} + \lVert u_2 \rVert_{\dot{B}_{\infty, \infty}^{\frac{1}{2}}}^{2}+\lVert u_3\rVert_{H^2}^{2}+1) \lVert u_{k} \rVert_{L^{2}}^{2}, 
\end{align} 
where we applied  H$\ddot{\mathrm{o}}$lder's inequality, Sobolev embedding $H^{2}(\mathbb{T}^2) \hookrightarrow L^{\infty}(\mathbb{T}^2)$, inequality \eqref{est 500k}, Lemma \ref{lemma_3.12k}, and Young's inequality, and Proposition \ref{prop_3.6k}.
\end{proof}

The  Proposition $\ref{prop_3.6k}$ and Proposition $\ref{Thm 3.28k}$ lead to the following theorem.

\begin{theorem}
\label{3.10Y}
Let $u_0 \in H^2(\mathbb{T}^3)$. If $ u_1,u_2 \in L^2(0,T;\dot{B}_{\infty, \infty}^{-\frac{1}{2}}(\mathbb{T}^3)) \cap L^2(0,T;  L^2(0,T;\dot{B}_{\infty, \infty}^{\frac{1}{2}}(\mathbb{T}^3))$,  then the  Kuramoto-Sivashinsky equation \eqref{KSE} admits a global strong solution, i.e., $u \in L^\infty([0, T]; H^2(\mathbb{T}^3))\cap L^2([0, T]; H^4(\mathbb{T}^3))$.
\end{theorem}
\begin{proof}
Taking $L^2 (\mathbb{T}^3)$-inner products on the  Kuramoto-Sivashinsky equation \eqref{KSE}  with $\Delta u$
\begin{align}
\frac{1}{2} \frac{d}{dt} \lVert& \nabla u \rVert_{L^2}^2+   \lVert \nabla \Delta u \rVert_{L^2}^2= - \int_{\mathbb{T}^3} (u\cdot\nabla) u \cdot \Delta u dx+\lambda \lVert \Delta u \rVert_{L^2}^{2}
\nonumber\\
&= \lambda \lVert \Delta u \rVert_{L^2}^{2} - \sum_{k=1}^{3} \int_{\mathbb{T}^3}
(u_1 \partial_1 u_1 \partial_k^2 u_1 +u_2 \partial_2 u_1 \partial_k^2 u_1+ u_3 \partial_3 u_1\partial_k^2 u_1+
\nonumber\\
&+u_1 \partial_1 u_2 \partial_k^2 u_2 +u_2 \partial_2 u_2\partial_k^2 u_2+ u_3 \partial_3 u_2\partial_k^2 u_2
+u_1 \partial_1 u_3 \partial_k^2 u_3 
\nonumber\\
&+u_2 \partial_2 u_3\partial_k^2 u_3+ u_3 \partial_3 u_3 \partial_k^2 u_3 ) dx,
\end{align}
after  applying  Proposition \ref{prop_3.6k} and Proposition \ref{Thm 3.28k}, Gr\"onwall's inequality implies $u \in L^\infty([0, T]; H^1(\mathbb{T}^3))\cap L^2([0, T]; H^3(\mathbb{T}^3))$. This completes the proof of the Theorem by bootstrapping for the higher regularity.
\end{proof}
\begin{remark}
It is worth noting that Theorem~\ref{3.10Y} is particularly interesting, as no cancellation in the nonlinear term occurs due to the lack of the divergence-free condition in the  Kuramoto-Sivashinsky equation \eqref{KSE}, in contrast to the Navier-Stokes equations; however, we establish estimates involving the velocity components $u_1$ and $u_2$.
s.
\end{remark}

\begin{theorem}
\label{3.12H}
Let $ u_0 \in H^2(\mathbb{T}^3)$. If $\nabla u \in L^2(0,T,\dot{B}_{\infty, \infty}^{-1}(\mathbb{T}^3)) \cap L^2(0,T; L^2(\mathbb{T}^3))$, then the  Kuramoto-Sivashinsky equation \eqref{KSE} admits a global strong solution, i.e., $u \in L^\infty([0, T]; H^2(\mathbb{T}^3))\cap L^2([0, T]; H^4(\mathbb{T}^3))$.
\end{theorem}
\begin{proof} 
Taking $L^2 (\mathbb{T}^3)$-inner products on the  Kuramoto-Sivashinsky equation \eqref{KSE}  with $\Delta u$
\begin{align}
\label{678K}
  \frac{1}{2} \frac{d}{dt} \lVert \nabla u \rVert_{L^2}^2+   \lVert \nabla \Delta u \rVert_{L^2}^2=& - \int_{\mathbb{T}^3} (u\cdot\nabla) u \cdot \Delta u dx+\lambda \lVert \Delta u \rVert_{L^2}^{2}
  \nonumber\\
  & \leq  \lVert \nabla u\rVert_{L^3}^{3}+ \lambda \lVert \Delta u\rVert_{L^2}^2
  \nonumber\\
    & \leq c(\lVert \nabla u \rVert_{\dot{B}_{\infty, \infty}^{-1}}^{\frac{1}{3}} \lVert \nabla u \rVert_{\dot{H}^{\frac{1}{2}}}^{\frac{2}{3}})^3+ \frac{1}{4}\lVert \nabla \Delta u\rVert_{L^2}^{2}+c\lVert \nabla u\rVert_{L^2}^2
    \nonumber\\
    & = c\lVert \nabla u \rVert_{\dot{B}_{\infty, \infty}^{-1}} \lVert \nabla u \rVert_{\dot{H}^{\frac{1}{2}}}^{{2}}+ \frac{1}{4}\lVert \nabla \Delta u\rVert_{L^2}^{2}+c\lVert \nabla u\rVert_{L^2}^2
    \nonumber\\
    &  \leq c \lVert \nabla u\rVert_{L^2} \lVert \nabla u \rVert_{\dot{B}_{\infty, \infty}^{-1}} \lVert \Delta u \rVert_{L^2}+ \frac{1}{4}\lVert \nabla \Delta u\rVert_{L^2}^{2}+c\lVert \nabla u\rVert_{L^2}^2
    \nonumber\\
    & \leq \frac{1}{8}\lVert \Delta u \rVert_{L^2}^{2}+c \lVert \nabla u \rVert_{\dot{B}_{\infty, \infty}^{-1}}^{2 } \lVert \nabla u\rVert_{L^2}^{2}+ \frac{1}{4}\lVert \nabla \Delta u\rVert_{L^2}^{2}+c\lVert \nabla u\rVert_{L^2}^2,
    \end{align}
where we applied integration by parts, H\"older's inequality, Lemma \ref{lemma_3.12k}, Young's inequality,  Gagliardo-Nirenberg interpolation inequality.

We also need an $L^2 (\mathbb{T}^3)$ estimate, thus taking $L^2 (\mathbb{T}^3)$-inner products on the  Kuramoto-Sivashinsky equation \eqref{KSE}  with $u$
    \begin{align}
    \label{K190H}
    \frac{1}{2} \frac{d}{dt} \lVert u \rVert_{L^{2}}^{2}+  \lVert \Delta u \rVert_{L^{2}}^{2} &=  - \int_{\mathbb{T}^3} (u\cdot\nabla) u \cdot u dx+\lambda \lVert \nabla u \rVert_{L^{2}}^{2}
    \nonumber\\
    & \leq  \lVert u \rVert_{L^6}\lVert \nabla u\rVert_{L^3}\lVert  u\rVert_{L^2}+c\lVert  u\rVert_{L^2}^2+ \frac{1}{4}\lVert \Delta u\rVert_{L^2}^{2}
    \nonumber\\
    & \leq c \lVert \nabla u \rVert_{L^2}\lVert \nabla \Lambda^{\frac{1}{2}}u\rVert_{L^2}\lVert  u\rVert_{L^2}+c\lVert  u\rVert_{L^2}^2+ \frac{1}{4}\lVert \Delta u\rVert_{L^2}^{2}
    \nonumber\\
    & \leq c \lVert \nabla u \rVert_{L^2}^{2}\lVert  u\rVert_{L^2}^{2}+ \frac{1}{4}\lVert \nabla \Lambda^{\frac{1}{2}}u\rVert_{L^2}^{2}+ c\lVert  u\rVert_{L^2}^2+\frac{1}{4}\lVert \Delta u\rVert_{L^2}^{2},
\end{align}
where we applied H\"older's inequality, Sobolev embedding and Young's inequality.  Applying \eqref{678K} and \eqref{K190H}, consequently, Gr\"onwall's inequality implies $u \in L^\infty([0, T]; H^1(\mathbb{T}^3))\cap L^2([0, T]; H^3(\mathbb{T}^3))$. This completes the proof of the Theorem by bootstrapping for the higher regularity.
\end{proof}
\begin{theorem}\label{Theorem 3.10kM}
Let $u_0 \in H^2(\mathbb{T}^2)$. If $\nabla \cdot  u \in L^2(0,T,\dot{B}_{2,2 }^{0}(\mathbb{T}^2))$, 
then the  Kuramoto-Sivashinsky equation \eqref{KSE} admits a global strong solution, i.e., $u \in L^\infty([0, T]; H^2(\mathbb{T}^2))\cap L^2([0, T]; H^4(\mathbb{T}^2))$.
\end{theorem}

\begin{proof}
We take $L^2 (\mathbb{T}^2)$-inner products on the  Kuramoto-Sivashinsky equation \eqref{KSE}  with $u$
    \begin{align}
    \label{est 3.17k}
    \frac{1}{2} \frac{d}{dt} \lVert u \rVert_{L^{2}}^{2}+  \lVert \Delta u \rVert_{L^{2}}^{2} &=  - \int_{\mathbb{T}^2} (u\cdot\nabla) u \cdot u dx+ \lambda \lVert \nabla u \rVert_{L^{2}}^{2}
    \nonumber\\
    & \leq  \lVert  u\rVert_{L^4}^2 \lVert \nabla \cdot u \rVert_{L^2}+c\lVert  u\rVert_{L^2}^2+ \frac{1}{4}\lVert \Delta u\rVert_{L^2}^{2}
    \nonumber\\
    & \leq c  \lVert \nabla \cdot  u \rVert_{\dot{B}_{2,2}^{0}}\lVert  u\rVert_{L^2}\lVert  \nabla u\rVert_{L^2}+ c\lVert  u\rVert_{L^2}^2+\frac{1}{4}\lVert \Delta u\rVert_{L^2}^{2}
    \nonumber\\
    & \leq c \lVert \nabla \cdot  u \rVert_{\dot{B}_{2,2}^{0}}^{2}\lVert  u\rVert_{L^2}^{2}+ \frac{1}{4}\lVert \nabla u\rVert_{L^2}^{2}+c\lVert  u\rVert_{L^2}^2+ \frac{1}{4}\lVert \Delta u\rVert_{L^2}^{2},
\end{align}
where we applied integration by parts, H\"older inequality, Lemma \ref{lemma_3.12k}, Young's inequality,  and Gr\"onwall's inequality to deduce $u \in L_t^{\infty} L_x^2 \cap L_t^{2}H_x^2$.

Taking $L^2 (\mathbb{T}^2)$-inner products on the  Kuramoto-Sivashinsky equation \eqref{KSE}  with $\Delta u$
\begin{align}
  \frac{1}{2} \frac{d}{dt} \lVert \nabla u \rVert_{L^2}^2+   \lVert \nabla \Delta u \rVert_{L^2}^2=& - \int_{\mathbb{T}^2} (u\cdot\nabla) u \cdot \Delta u dx+\lambda \lVert \Delta u \rVert_{L^2}^{2}
    \nonumber\\
    &\leq   \lVert \Delta u\rVert_{L^2} \lVert u\rVert_{L^{\infty}}  \lVert \nabla u\rVert_{L^2}+c\lVert \nabla u\rVert_{L^2}^2+ \frac{1}{4}\lVert \nabla \Delta u\rVert_{L^2}^{2}
    \nonumber\\
    & \leq \frac{1}{8}\lVert \Delta u \rVert_{L^2}^{2}+c \lVert \Delta u \rVert_{L^2}^{2 } \lVert \nabla u\rVert_{L^2}^{2}+c\lVert \nabla u\rVert_{L^2}^2+ \frac{1}{4}\lVert \nabla \Delta u\rVert_{L^2}^{2},
    \end{align}
where we applied  H\"older's inequality, Young's inequality, and, $u \in L_t^{\infty} L_x^2 \cap L_t^{2}H_x^2$ from \eqref{est 3.17k}, Gagliardo-Nirenberg interpolation inequality. Consequently, Gr\"onwall's inequality implies $u \in L^\infty([0, T]; H^1(\mathbb{T}^2))\cap L^2([0, T]; H^3(\mathbb{T}^2))$. This completes the proof of the Theorem by bootstrapping for the higher regularity.
\end{proof}
After rigorously establishing the regularity results, we now turn our attention to the error analysis for the Kuramoto-Sivashinsky equation \eqref{KSE} in the following section
\subsection{Error Analysis}
\label{3.2}
In this section,  we obtain PINNs approximations of the solutions to the Kuramoto-Sivashinsky equation \eqref{KSE} establishing the upper bound on the PINNs residual. Combining Corollary \ref{C3.2} and  Theorem \ref{2.10CT} located in the Appendix \ref{Appen} yield the following approximation result for the  Kuramoto-Sivashinsky equation \eqref{KSE}. 
\begin{theorem}
\label{3.11K}
Let $n \ge 2$, $d,m,k \in \mathbb{N}$ with $k \ge 5$, and let $u_0 \in H^m(\mathbb{T}^d)$ with
$m > \frac{d}{2} + 4k$. It holds that:
\begin{itemize}
\item there exists $T>0$ and a classical solution $u$ to the  Kuramoto-Sivashinsky equation \eqref{KSE}
such that $u \in H^k(\Omega)$, $\Omega = \mathbb{T}^d \times [0,T]$, and $u(t=0)=u_0$,
\item for every $N>6$, there exists  a tanh neural network $\overline u_j$, $1\le j\le d$,  with two hidden layers, of widths $ 
3\Big\lceil \tfrac{k+n-2}{2}\Big\rceil \binom{d+k-1}{d} + \lceil TN\rceil + dN$ and 
$3\Big\lceil \tfrac{d+n}{2}\Big\rceil \binom{2d+1}{d}\,\lceil TN\rceil\,N^d$
such that for every $1\le j\le d$,
\begin{subequations}
\begin{align}
\label{3.13A}
&\big\lVert (\overline u_j)_t + \overline u\cdot\nabla \overline u_j +\lambda \Delta \overline u_j+ \Delta^2 \overline u_j\big\rVert_{L^2(\Omega)}
\le C_1 \ln^4(\beta N)\,N^{-(k-4)},\\
&\|(u_0)_j-\overline u_j(t=0)\|_{L^2(\mathbb{T}^d)}\le C_2 \ln(\beta N)\,N^{-(k-1)}\label {3.13B}, 
\end{align}
\end{subequations}
where the constants $\beta, C_1, C_2$ are explicitly defined in the proof and can
depend on $k,d,T,u,$ but not on $N$. The weights of the networks can
be bounded by $O(N^{\gamma}\ln(N))$ where $\gamma=\max\{1,\, d(2 + k^2 + d)/n\}$.
\end{itemize}
\end{theorem}

\begin{proof}
Let $N>6$. By Corollary \ref{C3.2}, it holds that $u\in H^k(\mathbb{T}^d\times[0,T])$. As a result of Theorem \ref{2.10CT}, there then exists for every $1\le j\le d$ a tanh neural network $\overline u_j:=\overline u^N_j$ with two hidden layers and widths
\[
3\Big\lceil \tfrac{k+n-2}{2}\Big\rceil \binom{d+k-1}{d} + \lceil TN\rceil + dN
\quad\text{and}\quad
3\Big\lceil \tfrac{d+n}{2}\Big\rceil \binom{2d+1}{d}\,\lceil TN\rceil\,N^d
\]
such that for every $0\le \ell \le 4$,
\begin{equation}
\|u_j-\overline u_j\|_{H^\ell(\Omega)} \le C_{\ell,k,d+1,u_j}\,\lambda_\ell(N)\,N^{-k+\ell}, \label{eq:3.4}
\end{equation}
where $\lambda_\ell(N)=2^{\ell+1}3^{d}(1+\delta)\,\ln^{\ell}\!\big(\beta_{\ell,d+1,u_j}\,N^{\,d+k+2}\big)$, $\delta=\tfrac1{100}$, and the definition of
the other constants can be found in Theorem \ref{2.10CT}. The weights can be bounded
by $O(N^{\gamma}\ln(N))$ where $\gamma=\max\{1,\, d(2 + k^2 + d)/n\}$. We write $\overline u=(\overline u_1,\dots,\overline u_d)$.
We compute 
\begin{align}
\|(u_j)_t-(\overline u_j)_t\|_{L^2(\Omega)} \le \|u_j-\overline u_j\|_{H^1(\Omega)}\leq C_{1,k,d+1,u_j} \lambda_1(N)N^{(-(k-1))}.\label{4.21A}
\end{align}
By the Sobolev embedding theorem (Lemma \ref{2.8K}) it follows from $u\in C^1([0,T];H^{r-4}(\mathbb{T}^d))$
that $u\in C^1(\Omega)$, and hence
\begin{align}
\lVert u\cdot \nabla u_j-\overline{u}\cdot \nabla \overline{u}_j\rVert_{L^2(\Omega)}&=
\| u \cdot \nabla u_j - \overline{u} \cdot \nabla u_j + \overline{u} \cdot \nabla u_j - \overline{u} \cdot \nabla \overline{u}_j \|_{L^2(\Omega)}
\nonumber\\
&= \left\| \sum_{i=1}^d \left( u_i \partial_i u_j - \overline{u}_i \partial_i u_j + \overline{u}_i \partial_i u_j - \overline{u}_i \partial_i \overline{u}_j \right) \right\|_{L^2(\Omega)}
\nonumber\\
&\le  \| \sum_{i=1}^d(u_i - \overline{u}_i) \partial_i u_j \|_{L^2} 
   + \| \sum_{i=1}^d \overline{u}_i \, \partial_i (u_j - \overline{u}_j)\|_{L^2(\Omega)}
\nonumber\\
&\le \sum_{i=1}^d \| \partial_i u_j \|_{L^\infty(\Omega)} \| u_i - \overline{u}_i \|_{L^2(\Omega)}
   + \sum_{i=1}^d \| \overline{u}_i \|_{L^\infty} \| \partial_i (u_j - \overline{u}_j) \|_{L^2}
\nonumber\\
&\le d \, \lVert  \nabla u_j \rVert_{L^\infty(\Omega)} \max_i \lVert  u_i - \overline{u}_i \rVert_{L^2(\Omega)}+ d \max_{i}\lVert  \overline{u}_i \rVert_{L^\infty(\Omega)} \lVert  u_j - \overline{u}_j \rVert_{H^1(\Omega)}
\nonumber\\
&\le d \, \lVert  \nabla u_j \rVert_{C^1(\Omega)} \max_i \lVert  u_i - \overline{u}_i \rVert_{L^2(\Omega)}+ d \max_{i}\lVert  \overline{u}_i \rVert_{C^0(\Omega)} \lVert  u_j - \overline{u}_j \rVert_{H^1(\Omega)}
\nonumber\\
& \leq d \lambda_0(N)\lVert \nabla  u_j\rVert_{C^1(\Omega)}C_{0,k,d+1,u_j}N^{-k} +C_{1,k,d+1,u_j}\lambda_1(N) d \max_{i}\lVert \overline{u}_i\rVert_{C^0}N^{-(k-1)}.
\label{4.22A}
\end{align}
and 
\begin{align}
\lVert \Delta u_j - \Delta \overline u_j\rVert_{L^2(\Omega)} \le d \lVert u_j-\overline u_j\rVert_{H^2(\Omega)}
\leq d C_{2,k,d+1,u_j} \lambda_2(N)N^{(-(k-2))}\label{4.23A},
\end{align}
finally also
\begin{align}
\lVert \Delta^2 u_j - \Delta^2 \overline{u}_j \rVert_{L^2(\Omega)} \le C \lVert u_j - \overline{u}_j \rVert_{H^4(\Omega)}\leq d C_{4,k,d+1,u_j} \lambda_4(N)N^{(-(k-4))}
\label{eq:4.24A}.
\end{align}
Therefore, putting all the estimates \eqref{4.21A}-\eqref{eq:4.24A}, for $1\le j\le d$, we obtain
\begin{align}
\lVert (  (\overline u_j)_t + \overline u\cdot\nabla \overline u_j + \lambda \Delta \overline{u}_j  + \Delta^2 \overline u_j\rVert_{L^2(\Omega)}
&\le C_{1,k,d+1,u_j} \lambda_1(N)N^{(-(k-1))}+d \lambda_0(N)\lVert \nabla  u_j\rVert_{C^1(\Omega)}C_{0,k,d+1,u_j}N^{-k} 
\nonumber\\
&+C_{1,k,d+1,u_j}\lambda_1(N) d \max_{i}\lVert \hat{u}_i\rVert_{C^0}N^{-(k-1)}
+d C_{2,k,d+1,u_j} \lambda_2(N)N^{(-(k-2))}
\nonumber\\
&+ d C_{4,k,d+1,u_j} \lambda_4(N)N^{(-(k-4))}.
\label{4.25A}
\end{align}
Finally, we find from the multiplicative trace theorem (Lemma \ref{MTI}) that
\begin{align}
\|(u_0)_j - \overline u_j(t=0)\|_{L^2(\mathbb{T}^d)}
&\le \|u_j - \overline u_j\|_{L^2(\partial \Omega)}
\le \sqrt{\frac{2\max\{2h_\Omega,\, d+1\}}{\rho_\Omega}}\;\|u_j-\overline u_j\|_{H^1(\Omega)} \notag\\
&\le \sqrt{\frac{2\max\{2h_\Omega,\, d+1\}}{\rho_\Omega}}\;C_{1,k,d+1,u_1}\,\lambda_1(N)\,N^{-k+1}, \label{eq:3.12}
\end{align}
where $h_\Omega$ is the diameter of $\Omega$ and $\rho_\Omega$ is the radius of the largest $(d+1)$-dimensional
ball, which concludes the proof.
\end{proof}
\begin{remark}
The bounds \eqref{3.13A}-\eqref{3.13B} in Theorem \ref{3.11K} affirmatively answer Question Q1 when $N \to \infty$, by which we can make PINN residuals \eqref{2.5PL} and the generalization error \eqref{2.6H} arbitrarily small.
\end{remark}
We next establish an $L^2$-error bound between the true solution $u$ and the PINN approximation $u^{\ast}$ to the Kuramoto-Sivashinsky equation \eqref{KSE} on the periodic domain $D = [0,2]^d$ and $n$ denotes the unit normal vector to the boundary, $\partial D$. We define the following PINNs-related residuals:
\begin{equation}
\begin{cases}
\begin{aligned}
\mathcal{R}_{\mathrm{PDE}} &= u_t^{\ast} + (u^{\ast}\!\cdot\!\nabla)u^{\ast} + \Delta u^{\ast} + \Delta^2 u^{\ast}, 
&& (x,t)\in D\times(0,T),\\
\mathcal{R}_{\text{sb}}^{1} &= u^{\ast}(x)-u^{\ast}(x+2),
&& (x,t)\in\partial D\times(0,T),\\
\mathcal{R}_{\text{sb}}^{2} &= \partial_n(u^{\ast}(x)-u^{\ast}(x+2)),
&& (x,t)\in\partial D\times(0,T),\\
\mathcal{R}_{\text{sb}}^{3} &= \Delta u^{\ast}(x)-\Delta u^{\ast}(x+2),
&& (x,t)\in\partial D\times(0,T),\\
\mathcal{R}_{\text{sb}}^{4} &= \partial_n(\Delta u^{\ast}(x)-\Delta u^{\ast}(x+2)),
&& (x,t)\in\partial D\times(0,T),\\
\mathcal{R}_{\text{tb}}&=u^{\ast}(t=0)-u(t=0), 
&& x\in D.
\end{aligned}
\end{cases}
\label{234A}
\end{equation}
Using PINN-residuls \eqref{234A} together with  the  Kuramoto-Sivashinsky equation \eqref{KSE}, we state the following result. 
\begin{theorem} 
\label{Th 4.3}
Let $u \in \mathbb{C}^4(\bar{D} \times [0,T])$ be the  solution of the Kuramoto-Sivashinsky equation \eqref{KSE}. Let $u^* \in \mathbb{C}^4(\bar{D} \times [0,T])$ is  the PINNs generated solution, then the resulting $L^2$-error is bounded by,
\begin{align}
\lVert u^{\ast}(x,t)- u(x,t)\rVert_{L^2(\Omega)}^{2} \leq C_3(T+C_{2}T^2e^{C_{2}T}),
\end{align}
where 
\begin{align}
C=C(\lVert u\rVert_{\mathbb{C}(\Omega)}^4,\lVert \hat{u}\rVert_{\mathbb{C}(\Omega)}^4 ),
\end{align}
and
\begin{align}
C_2= C_1(\lVert \hat{u}\rVert_{L^{\infty}(D)}^{2}+ \lVert u \rVert_{L^{\infty}(D)}^{2}+ \lVert \nabla u\rVert_{L^2(D)}^{2}+1),
\end{align}
where $C_1>0$ arising from Young’s inequality, 
and
\begin{align*}
C_3&=\lVert\mathcal{R}_{tb}\rVert_{L^2(\partial D)}^{2}+ C\lVert \mathcal{R}_{\textrm{PDE}}\rVert_{L^2(\Omega)}
\nonumber\\
&+C \sqrt{\partial D}T^{1/2}\big[\lVert \mathcal{R}_{sb}^{1}\rVert_{L^2(\partial D\times [0,T])}+\lVert \mathcal{R}_{sb}^{2}\rVert_{L^2(\partial D\times [0,T])}+\lVert \mathcal{R}_{sb}^{3}\rVert_{L^2(\partial D\times [0,T])}+\lVert \mathcal{R}_{sb}^{4}\rVert_{L^2(\partial D\times [0,T])}\big],
\end{align*}
with $\Omega= D \times [0,T]$.
\end{theorem} 
\begin{proof}    
We denote the difference between the underlying solution \( u \) and PINN \( u^* \) as \( \hat{u} = u^* - u \):
\begin{align}
&\mathcal{R}_{PDE}=\hat{u}_t + (\hat{u} \cdot \nabla) \hat{u} + (u \cdot \nabla) \hat{u} + (\hat{u} \cdot \nabla) u+\Delta \hat{u}+\Delta^2 \hat{u}, \quad (x, t) \in D \times (0, T),
\nonumber\\
&
\mathcal{R}_{\text{sb}}^{1}= u^*(x)-u^*(x+2), \quad (x, t) \in \partial D \times (0, T),
\nonumber\\
&
\mathcal{R}_{\text{sb}}^{2}= \partial_n(u^*(x)-u^*(x+2)), \quad (x, t) \in \partial D \times (0, T),
\nonumber\\
& \mathcal{R}_{\text{sb}}^{3}= \Delta u^*(x)-\Delta u^*(x+2), \quad (x, t) \in \partial D \times (0, T),
\nonumber\\
&\mathcal{R}_{\text{sb}}^{4}=\partial_n( \Delta u^*(x)-\Delta u^*(x+2)), \quad (x, t) \in \partial D \times (0, T),
\nonumber\\
&
\mathcal{R}_{\text{tb}}=\hat{u}(x, 0), \quad x \in D.
\end{align}
We take $L^2$-inner products of the first equation in  with the vector \( \hat{u} \)  yields the following identity:
\begin{align}
\label{4.17A}
\frac{1}{2}\frac{d}{dt} \int_{D}  \lvert \hat{u}\rvert^2 dx 
+\int_{D}\lvert \Delta \hat{u}\rvert^2 dx
&=\int_{\partial D}\Delta \hat{u}\cdot \frac{\partial \hat{u}}{\partial n} ds-\int_{\partial D}\frac{\partial (\Delta \hat{u})}{\partial n}\cdot \hat{u} ds-\int_{D}(\hat{u} \cdot \nabla) \hat{u}\cdot\hat{u}  dx - \int_{D}(u \cdot \nabla) \hat{u}\cdot\hat{u} dx
\nonumber\\
&- \int_{D} (\hat{u}\cdot \nabla)u\cdot \hat{u} dx+\int_{D} \nabla \hat{u} \cdot \nabla \hat{u} dx-\int_{\partial D}\frac{\partial \hat{u}}{\partial n} \cdot \hat{u}+ \int_{D} \hat{u} \cdot \mathcal{R}_{PDE} dx.
\end{align}
Applying Hölder's, and Young's inequalities implies the following inequality:
\begin{align}
\label{4.18A}
&\int_{D}(\hat{u} \cdot \nabla) \hat{u}\cdot\hat{u}  dx
\nonumber\\
& \leq \lVert \hat{u}\rVert_{L^{\infty}(D)}\lVert \nabla \hat{u}\rVert_{L^2(D)}\lVert \hat{u}\rVert_{L^2(D)}
\nonumber\\
& \leq C_1 \lVert \hat{u}\rVert_{L^{\infty}(D)}^{2}\lVert\hat{u}\rVert_{L^2(D)}^{2}+\frac{1}{4}\lVert \nabla \hat{u}\rVert_{L^2(D)}^{2}.
\end{align}
Applying Hölder's, and Young's inequalities implies the following inequality:
\begin{align}
\int_{D}(u \cdot \nabla) \hat{u}\cdot\hat{u} dx&\leq \lVert u\rVert_{L^{\infty}(D)}\lVert \nabla \hat{u}\rVert_{L^2(D)}\lVert \hat{u}\rVert_{L^2(D)}
\nonumber\\
&\leq C_1\lVert u\rVert_{L^{\infty}(D)}^{2}\lVert \hat{u}\rVert_{L^2(D)}^{2}+\frac{1}{4}\lVert \nabla \hat{u}\rVert_{L^2(D)}^{2}.
\end{align}
Applying Hölder's, Gagliardo-Nirenberg,  and Young's inequalities implies the following inequality:
\begin{align}
\int_D  ((\hat{u} \cdot \nabla) u)\cdot\hat{u}\, dx &\leq \lVert \nabla u\rVert_{L^2(D)}  \lVert \hat{u}\rVert_{L^4(D)}^{2}
\nonumber\\
&\leq  \lVert \nabla u\rVert_{L^2(D)}\lVert \hat{u}\rVert_{L^2(D)}\lVert \nabla \hat{u}\rVert_{L^2(D)}
\nonumber\\
&\leq C_1 \lVert \nabla u\rVert_{L^2(D)}^{2}  \lVert \hat{u}\rVert_{L^2(D)}^{2}+ \frac{1}{4}\lVert \nabla \hat{u}\rVert_{L^2(D)}^{2}. 
\end{align}
Applying Lemma \ref{est 500k}, Young's Inequality, we obtain in the following:
\begin{align}
\label{4.26F}
 \| \nabla u \|_{L^2(D)}^2 \leq \frac{1}{4} \| \Delta u \|_{L^2(D)}^2 + C_1 \| u \|_{L^2(D)}^2.
\end{align}
Putting all estimates \eqref{4.18A} to \eqref{4.26F} into \eqref{4.17A} yields the following:
\begin{align}
\frac{1}{2}\frac{d}{dt} \int_{D}  \lvert \hat{u}\rvert^2 dx 
+&\int_{D}\lvert \Delta \hat{u}\rvert^2 dx
=\int_{\partial D}\Delta \hat{u}\cdot \frac{\partial \hat{u}}{\partial n} ds-\int_{\partial D}\frac{\partial (\Delta \hat{u})}{\partial n}\cdot \hat{u} ds-\int_{D}(\hat{u} \cdot \nabla) \hat{u}\cdot\hat{u}  dx - \int_{D}(u \cdot \nabla) \hat{u}\cdot\hat{u} dx
\nonumber\\
&- \int_{D} (\hat{u}\cdot \nabla)u\cdot \hat{u} dx+\int_{D} \nabla \hat{u} \cdot \nabla \hat{u} dx-\int_{\partial D}\frac{\partial \hat{u}}{\partial n} \cdot \hat{u}+ \int_{D} \hat{u} \cdot \mathcal{R}_{PDE} dx.
\nonumber\\
&
\leq  C(\lVert u\rVert_{\mathbb{C}(D)}^4,\lVert \hat{u}\rVert_{\mathbb{C}(D)}^4 ) (\lVert \mathcal{R}_{sb}^{1}\rVert_{L^2(\partial D)} +  \lVert \mathcal{R}_{sb}^2\rVert_{L^2(\partial D)}+\lVert \mathcal{R}_{sb}^3\rVert_{L^2(\partial D)}+\lVert \mathcal{R}_{sb}^4\rVert_{L^2(\partial D)})
\nonumber\\
&+ C_1 \lVert \hat{u}\rVert_{L^{\infty}(D)}^{2} \lVert \hat{u}\rVert_{L^2(D)}^2 + \frac{1}{4} \lVert \nabla \hat{u}\rVert_{L^2(D)}^2
\nonumber\\
&+ C_1 \lVert u\rVert_{L^{\infty}(D)}^{2} \lVert \hat{u}\rVert_{L^2(D)}^2 + \frac{1}{4} \lVert \nabla \hat{u}\rVert_{L^2(D)}^2
+ C_1 \lVert \nabla u\rVert_{L^2(D)}^{2} \lVert \hat{u}\rVert_{L^2(D)}^{2}+ \frac{1}{4} \lVert \nabla \hat{u}\rVert_{L^2(D)}^2 \nonumber\\
&+C(\lVert u\rVert_{\mathbb{C}(D)}^4,\lVert \hat{u}\rVert_{\mathbb{C}(D)}^4 )(\lVert \mathcal{R}_{sb}^{1}\rVert_{L^2(\partial D)}+\lVert \mathcal{R}_{sb}^{2}\rVert_{L^2(\partial D)}  )+ C(\lVert u\rVert_{\mathbb{C}(D)}^4,\lVert \hat{u}\rVert_{\mathbb{C}(D)}^4 ) \lVert \mathcal{R}_{PDE}\rVert_{L^2(D)}
\nonumber\\&+ \frac{1}{4} \lVert  \Delta u \rVert_{L^2(D)}^2+ C_1 \lVert  u \rVert_{L^2(D)}^2
\nonumber\\
&\leq C(\lVert u\rVert_{\mathbb{C}(D)}^4,\lVert \hat{u}\rVert_{\mathbb{C}(D)}^4 ) (\lVert \mathcal{R}_{sb}^1\rVert_{L^2(\partial D)} +  \lVert \mathcal{R}_{sb}^2\rVert_{L^2(\partial D)}+\lVert \mathcal{R}_{sb}^3\rVert_{L^2(\partial D)}+\lVert \mathcal{R}_{sb}^4\rVert_{L^2(\partial D)})
\nonumber\\
&+ C_2 \lVert \hat{u}\rVert_{L^2(D)}^2 + \frac{1}{2} \lVert \Delta \hat{u}\rVert_{L^2(D)}^2 + C(\lVert u\rVert_{\mathbb{C}(D)}^4,\lVert \hat{u}\rVert_{\mathbb{C}(D)}^4 ) \lVert \mathcal{R}_{PDE}\rVert_{L^2(D)},
\label{4.23D}
\end{align}
where we used H\"older's inequality, Young's inequality, trace inequality, and the assumption on $u \in \mathbb{C}^4(\bar{D} \times [0,T])$ and $u^{\ast} \in \mathbb{C}^4(\bar{D} \times [0,T])$ with the traingle inequality allows us to close the estimate. 
Absorbing the corresponding term with left side, integrating over $[0,T]$, and dropping the second term from the left side of \eqref{4.23D} implies
\begin{align}
\int_{D}\lvert \hat{u}\rvert^2 dx&\leq  C_2 \int_{0}^{T}\lVert \hat{u}\rVert_{L^2(D)}^2 dt
+ C(\lVert u\rVert_{\mathbb{C}(\Omega)}^4,\lVert \hat{u}\rVert_{\mathbb{C}(\Omega)}^4 )\int_{0}^{T}(\lVert \mathcal{R}_{sb}^1\rVert_{L^2(\partial D)} +  \lVert \mathcal{R}_{sb}^2\rVert_{L^2(\partial D)}\nonumber\\
&+\lVert \mathcal{R}_{sb}^3\rVert_{L^2(\partial D)}+\lVert \mathcal{R}_{sb}^4\rVert_{L^2(\partial D)}) dt
+C(\lVert u\rVert_{\mathbb{C}(\Omega)}^4,\lVert \hat{u}\rVert_{\mathbb{C}(\Omega)}^4 )\int_{0}^{T}\lVert \mathcal{R}_{PDE}\rVert_{L^2(D)} dt+\lVert \mathcal{R}_{tb}\rVert_{L^2(\partial D)}^{2} \nonumber\\
&\leq  C_2 \int_{0}^{T}\lVert \hat{u}\rVert_{L^2(D)}^2 dt+C_3,
\end{align}
with 
\begin{align}
C_2= C_1(\lVert \hat{u}\rVert_{L^{\infty}(D)}^{2}+ \lVert u \rVert_{L^{\infty}(D)}^{2}+ \lVert \nabla u\rVert_{L^2(D)}^{2}+1),
\end{align}
and
\begin{align*}
C_3&=\lVert\mathcal{R}_{tb}\rVert_{L^2(\partial D)}^{2}+ C\lVert \mathcal{R}_{\textrm{PDE}}\rVert_{L^2(\Omega)}
\nonumber\\
&+C \sqrt{\partial D}T^{1/2}\big[\lVert \mathcal{R}_{sb}^{1}\rVert_{L^2(\partial D\times [0,T])}+\lVert \mathcal{R}_{sb}^{2}\rVert_{L^2(\partial D\times [0,T])}+\lVert \mathcal{R}_{sb}^{3}\rVert_{L^2(\partial D\times [0,T])}+\lVert \mathcal{R}_{sb}^{4}\rVert_{L^2(\partial D\times [0,T])}\big], 
\end{align*} and
applying Gr\"onwall's inequality  yields,
\begin{align}
\lVert \hat{u}\rVert_{L^2(D)}^{2}\leq C_3(1+C_2Te^{C_2 T}).
\end{align}
Integrating from $0$ to $T$ implies in the following
\begin{align}
\label{4.25K}
\lVert u^{\ast}(x,t)- u(x,t)\rVert_{L^2(\Omega)}^{2} \leq C_3(T+C_{2}T^2e^{C_{2}T}).
\end{align}
\end{proof}
\begin{remark}
 To answer Question Q2 affirmatively, we note that for PINNs, the $L^2$-error is uniquely bounded in terms of the residuals that constitute the PINN generalization error \eqref{2.6H}. Consequently, if the neural network achieves a small PINN loss and its $\mathbb{C}^4$-norm remains bounded, the corresponding $L^2$-error will also be small. 
\end{remark}
We now turn our focus to addressing Question Q3. By combining the following result \ref{4.4} with Theorem \ref{Th 4.3}, we obtain a bound for the total error in terms of the training error and the size of the training set, thereby providing an affirmative answer to Question Q3.
\begin{theorem}
\label{4.4}
Let $T > 0$, let $u \in \mathbb{C}^6(\overline{D} \times [0,T])$ be the classical solution of the Kuramoto-Sivashinsky equation \eqref{KSE} and let $u^{\ast} \in \mathbb{C}^6(\overline{D} \times [0,T])$ is the PINNs generated solution. Then the following error bound holds,
\begin{align}
\label{4.26W}
\lVert u^{\ast}(x,t)- u(x,t)\rVert_{L^2(\Omega)}^{2}&\le\; C_3(M)(T+C_{2}T^2e^{C_{2}T}),
\nonumber\\
&= \mathcal{O}\Bigl(E_T(S) + M_t^{-2/d} + M_{\mathrm{int}}^{-1/(d+1)} + M_{sb,1}^{-1/d}+ M_{sb,2}^{-1/d}+M_{sb,3}^{-1/d}+ M_{sb,4}^{-1/d} \Bigr).
\end{align}
In the above formula, the constant $C_2$ and $C_3(M)$ are defined as,
\begin{align}
C_2= C_1(\lVert \hat{u}\rVert_{L^{\infty}(\Omega)}^{2}+ \lVert u \rVert_{L^{\infty}(\Omega)}^{2}+ \lVert \nabla u\rVert_{L^2(\Omega)}^{2}+1),
\end{align} and
\begin{align}
C_3(M) &= E_T^t(S_t)^2 + C_t M_t^{-2/d}
+ E_T^{\mathrm{PDE}}(\theta,S_{\mathrm{int}})^2
+ C_{\mathrm{PDE}} M_{\mathrm{int}}^{-2/(d+1)} +CT^{1/2}\big[E_T^{sb,1}(\theta, S_{sb,1})+C_{sb,1}M_{sb,1}^{-\frac{1}{d}}
\nonumber\\
&
+E_T^{sb,2}(\theta, S_{sb,2})+ C_{sb,2}M_{sb,2}^{-\frac{1}{d}}+E_T^{sb,3}(\theta, S_{sb,3})+C_{sb,3}M_{sb,3}^{-\frac{1}{d}}+E_T^{sb,4}(\theta, S_{sb,4})+ C_{sb,4}M_{sb,4}^{-\frac{1}{d}}\big].
\end{align}
and where,
\begin{align}
&C \lesssim \|u\|_{C^4} + \|\hat u\rVert_{C^4}
\;\lesssim\;  \lVert u\rVert_{C^4}+16^{4L}(d+1)^8 (4^4 e^{2} W^3 R^4 \|\sigma\|_{C^4})^{4L},
\nonumber\\
&C_t \lesssim  \|u\|_{C^2}^2+\|\hat u\|_{C^2}^2
\;\lesssim\; \|u\|_{C^2}^2 + (e^{2}2^{6} W^3 R^2 \|\sigma\|_{C^2})^{2L}, 
\nonumber\\
&C_{\mathrm{PDE}}^{2} \lesssim \|\hat u_j\|_{C^6}^2
\;\lesssim\; ( e^{2}{6^4} W^3 R^6 \|\sigma\|_{C^6})^{6L},
\nonumber\\
&C_{sb,1}^2 \lesssim \lVert \hat u_j\rVert_{C^2}^{2}
\;\lesssim\; ( e^{2}2^{4} W^3 R^2 \lVert \sigma\rVert_{C^2})^{2L},
\nonumber\\
&C_{sb,2}^2 \lesssim \lVert \hat u_j\rVert_{C^3}^{2}
\;\lesssim\; ( e^{2}3^{4} W^3 R^3 \lVert \sigma\rVert_{C^3})^{3L},
\nonumber\\
& C_{sb,3}^2 \lesssim \lVert \hat u_j\rVert_{C^4}^{2}
\;\lesssim\; ( e^{2}4^{4} W^3 R^4 \lVert \sigma\rVert_{C^4})^{4L},
\nonumber\\
& C_{sb,4}^2 \lesssim \lVert \hat u_j\rVert_{C^5}^{2}
\;\lesssim\; ( e^{2}5^{4} W^3 R^5 \lVert \sigma\rVert_{C^5})^{5L}.
\end{align}
\label{4.30K}
\end{theorem}
\begin{proof}
The proof is straightforward and follows directly from Theorem \ref{Th 4.3}, the quadrature error bound \eqref{quad1}, and Lemma \ref{1KO}.  
\end{proof}
We next investigate the numerical experiments to illustrate PINN approximation of the  Kuramoto-Sivashinsky equation \eqref{KSE}.
\section{Numerical Experiments}
\label{sec4}
In this section, we present numerical experiments designed to empirically validate the theoretical error bounds derived for the PINN approximations of the Kuramoto-Sivashinsky equation. 

We consider the two-dimensional Kuramoto-Sivashinsky equation, for which a manufactured solution is defined as
\begin{align*}
u(x,y,t) = -\cos(\pi x)\sin(\pi y)\,
\exp\!\left(-\frac{\pi^2}{4}\lambda t\right),
\qquad
v(x,y,t) = \sin(\pi x)\cos(\pi y)\,
\exp\!\left(-\frac{\pi^2}{4}\lambda t\right),
\end{align*}
on the spatio-temporal domain \((x,y)\in[0,2]^2\) and \(t\in[0,1]\).
Using the manufactured solution as a reference, we approximate the solution of the two-dimensional Kuramoto-Sivashinsky equation within the PINN framework. 
The training is conducted in two stages: the first stage consists of 15K Adam and 5K L-BFGS iterations with a learning rate of \(1.0 \times 10^{-4}\); the second stage consists of 5K Adam and 5K L-BFGS iterations with a learning rate of \(2.5 \times 10^{-5}\). The PINN architecture employs two hidden layers with 80 neurons per layer and the hyperbolic tangent activation function, consistent with the theoretical setting. Based on this setup, Figures~\ref{fig:combined} and \ref{fig:three_v_cases} illustrate the distributions of the manufactured (exact) and PINN solutions, along with the absolute errors for $u$ and $v$ in the spatial-temporal domain, computed using $10{,}000$ residual quadrature points. 

We next  investigate the influence of the number of quadrature points on both the training and total errors. All experiments employ the same neural network configuration, consisting of two hidden layers with 80 neurons per layer and the hyperbolic tangent activation function.
To further elucidate the theoretical estimates, we evaluate the computable error bound in~\eqref{4.26W}. Using the explicit expressions for the constants provided in~\eqref{4.30K}, we compute the theoretical upper bound in~\eqref{4.26W}, where the following quadrature points are employed:
\[
M_t = 1000, \qquad
M_{\mathrm{int}} = \{10000, 15000, 20000\}, \qquad
M_{sb,1} = M_{sb,2} = M_{sb,3} = M_{sb,4} = 1000.
\]
Here, $M_t$ denotes the number of temporal quadrature points associated with the initial condition, 
$M_{\mathrm{int}}$ corresponds to the number of interior (PDE residual) points, 
and $M_{sb,1}, M_{sb,2}, M_{sb,3}, M_{sb,4}$ denote the number of quadrature points used for the boundary conditions. 
 \begin{figure}[H]
    \centering
    \begin{minipage}[b]{0.45\linewidth}
        \centering
        \includegraphics[width=\linewidth]{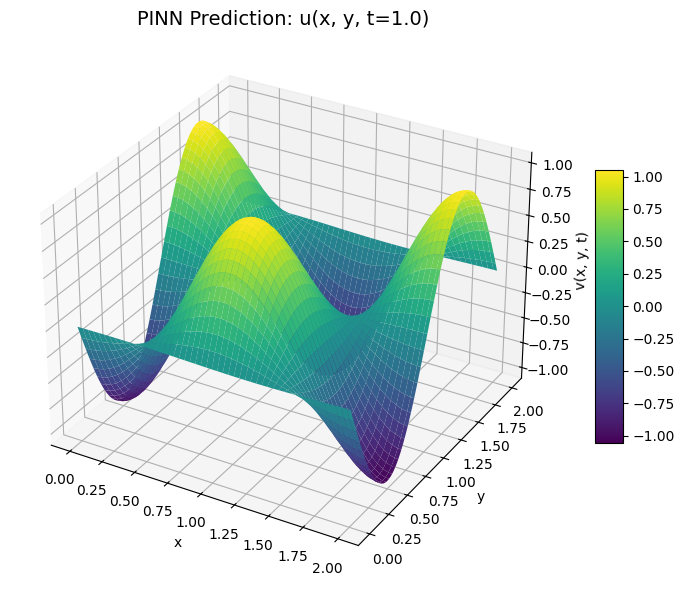}
    \end{minipage}
    \hfill
    \begin{minipage}[b]{0.45\linewidth}
        \centering
        \includegraphics[width=\linewidth]{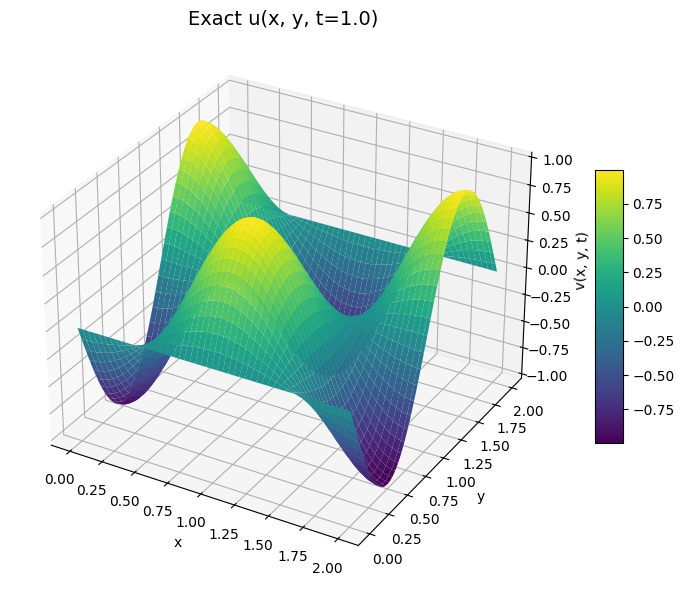}
    \end{minipage}

    \vspace{1em} 

    \begin{minipage}[b]{0.45\linewidth}
        \centering
\includegraphics[width=\linewidth]{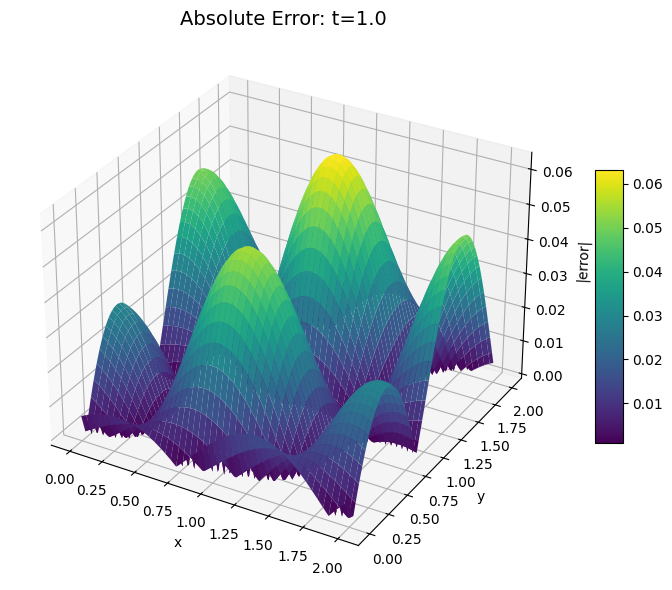}
    \end{minipage}

    \caption{PINNs Solution,  
    Manufacture (Exact) Solution, and Absolute Error}
\label{fig:combined}
\end{figure}
Keeping the number of neurons fixed, we compute the bound while varying the residual quadrature points $\{10000, 15000, 20000\}$ for each $\lambda \in \{0.0, 0.01\}$. 
The total error $\lVert u(x,t) - u^{\ast}(x,t)\rVert_{L^2}$, training error  $E_T(\theta,S)$ in \eqref{34LK}, along with the bound in~\eqref{4.26W}, are presented in \textbf{Figure~\ref{fig:combined_56}} for each $\lambda \in \{0.0, 0.01\}$.

 The results indicate that while the total error $\lVert u(x,t) - u^{\ast}(x,t)\rVert_{L^2}$, training error  $E_T(\theta,S)$ in \eqref{34LK} , and theoretical bound in \eqref{4.26W} decrease marginally with increasing residual points, this trend plateaus; and no further significant decay is observed. This plateau behavior is consistent with the logarithmic factors in Theorem \ref{3.11K}, which dominates the decay rate for large $N$, and suggests that the quadrature error is no longer the bottleneck. \begin{figure}[H]
    \centering
    \begin{minipage}[b]{0.45\linewidth}
        \centering
        \includegraphics[width=\linewidth]{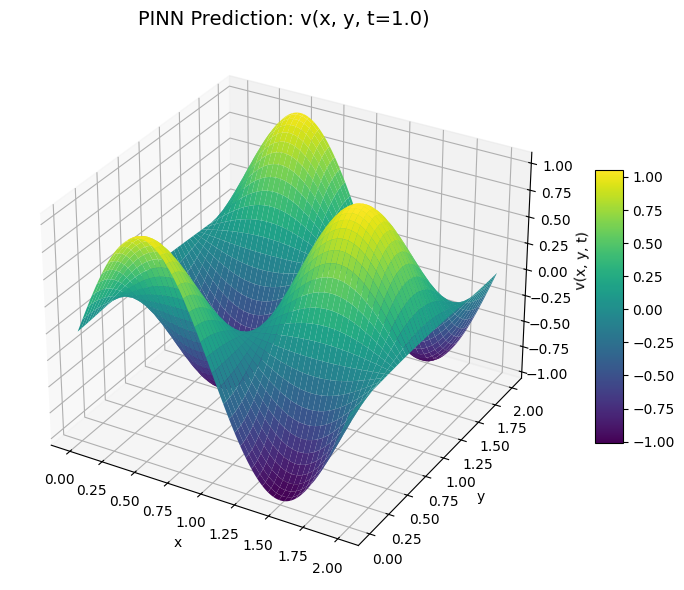}
    \end{minipage}
    \hfill
    \begin{minipage}[b]{0.45\linewidth}
        \centering
        \includegraphics[width=\linewidth]{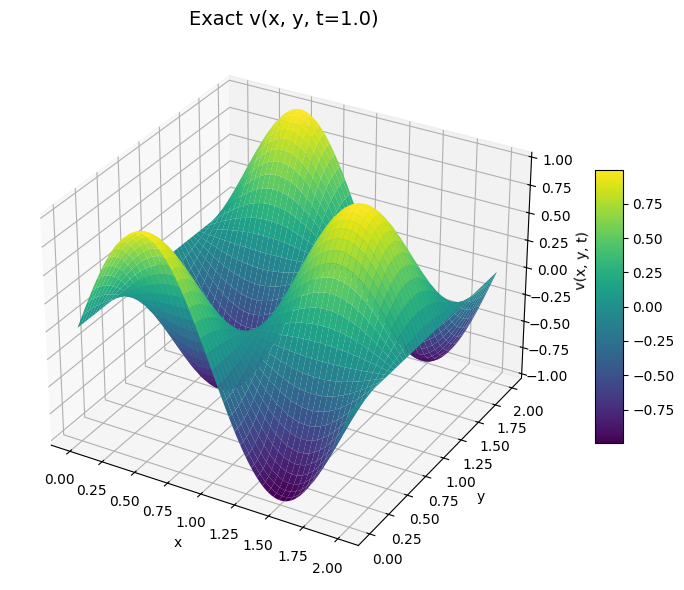}
    \end{minipage}

    \vspace{1em} 

    \begin{minipage}[b]{0.45\linewidth}
        \centering
        \includegraphics[width=\linewidth]{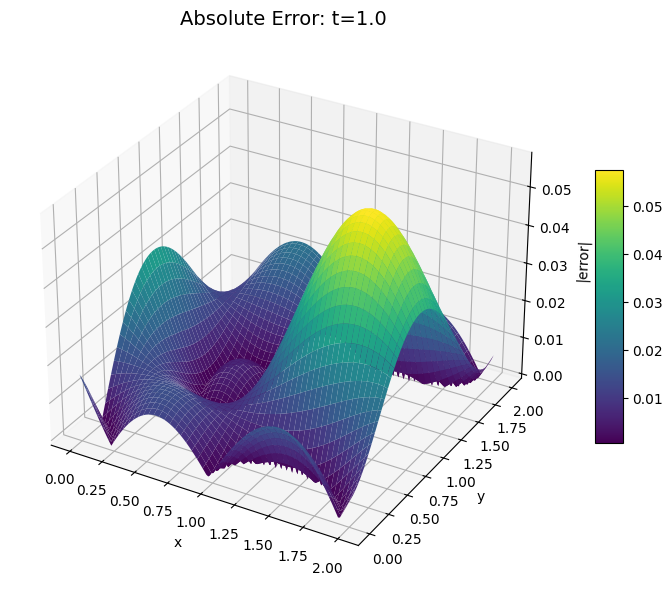}
    \end{minipage}

    \caption{PINNs Solution, Manufacture (Exact) Solution, and Absolute Error}
    \label{fig:three_v_cases}
\end{figure}

\begin{figure}[H]
    \centering
    \begin{minipage}[b]{0.45\linewidth}
        \centering
        \includegraphics[width=\linewidth]{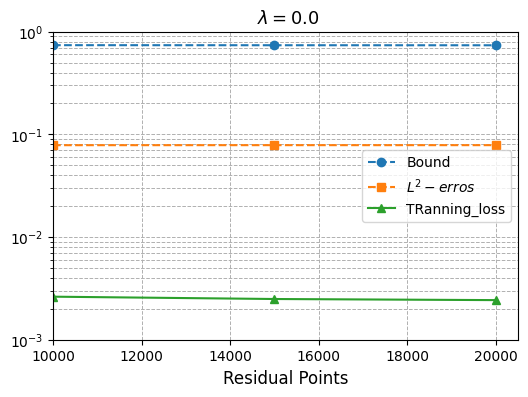}
    \end{minipage}
    \hfill
    \begin{minipage}[b]{0.45\linewidth}
        \centering
        \includegraphics[width=\linewidth]{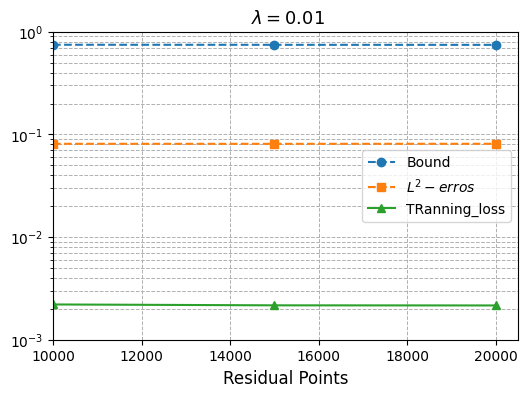}
    \end{minipage}

    \caption{Training error, total error, and bound for different numbers of residual points}
    \label{fig:combined_56}
\end{figure}
The results show that as the number of neurons increases, the total error $\lVert u(x,t) - u^{\ast}(x,t)\rVert_{L^2}$, training error  $E_T(\theta,S)$ in \eqref{34LK}, and bound in \eqref{4.26W} exhibit a slight decrease, followed by a plateau beyond which no further significant decay is observed. According to equation \eqref{4.26W}, the total error depends on both the neural network approximation capacity and the quadrature accuracy. The observed plateau suggests that for the fixed number of quadrature points ($10{,}000$), the quadrature error dominates, making further increases in network width less effective.
\begin{figure}[H]
    \centering
    \begin{minipage}[b]{0.45\linewidth}
        \centering
        \includegraphics[width=\linewidth]{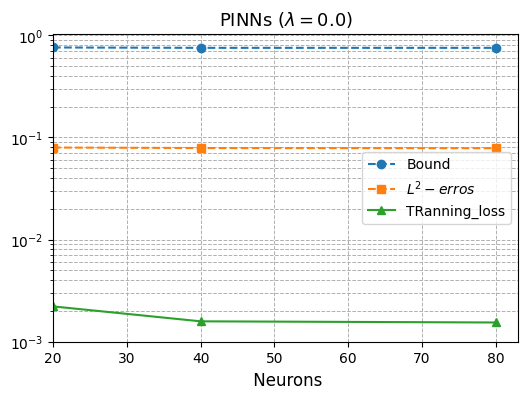}
    \end{minipage}
    \hfill
    \begin{minipage}[b]{0.45\linewidth}
        \centering
        \includegraphics[width=\linewidth]{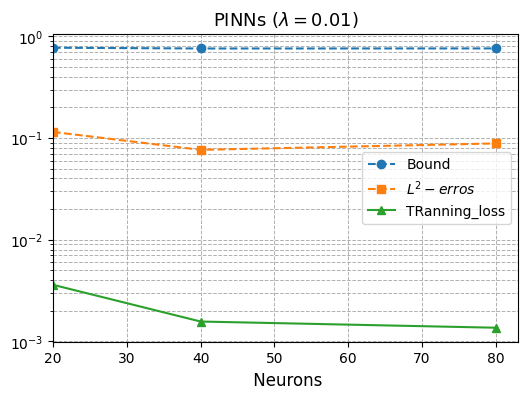}
    \end{minipage}

    \caption{Training error, total error, and bound for different numbers of neurons}
    \label{fig:combined_34}
\end{figure}
Finally, we investigate the dependence of the total error on the neural network width. Fixing the number of hidden layers at two, we vary the number of neurons $\{20, 40, 80\}$ per layer. The learning rate and the number of quadrature points $10000$ remain fixed. The total error $\lVert u(x,t) - u^{\ast}(x,t)\rVert_{L^2}$, training error  $E_T(\theta,S)$ in \eqref{34LK}, and bound in \eqref{4.26W} are shown in Figure
\ref{fig:combined_34} for each $\lambda=\{0.0,0.01\}$.
\section{Appendix}
\label{Appen}
\begin{theorem}
\label{2.1H}
If $u_0 \in H^m(\mathbb{T}^d)$ with $m > \frac{d}{2} + 4$, then there
exists $T > 0$ and a classical solution $u$ to the  Kuramoto-Sivashinsky equation \eqref{KSE} such that
$u(t=0) = u_0$ and $u \in C([0,T]; H^m(\mathbb{T}^d)) \cap C^1([0,T]; H^{m-4}(\mathbb{T}^d))$.
\end{theorem}
\begin{proof}
We apply $\Lambda^m$ to the  Kuramoto-Sivashinsky equation \eqref{KSE} and take $L^2(\mathbb{T}^d)$-inner products with $\Lambda^m u$
to obtain
\begin{align}
\frac{1}{2} \frac{d}{dt} \lVert u\rVert_{\dot{H}^m}^2 + \lVert u\rVert_{\dot{H}^{m+2}}^2
&= \int_{\mathbb{T}^2} \Lambda^m (u \cdot \nabla) u \cdot \Lambda^m u \, dx + \lVert u\rVert_{\dot{H}^{m+1}}^2 
\nonumber\\
&\le \lVert \Lambda^m(u \cdot \nabla)u\rVert_{L^2} \, \lVert \Lambda^m u\rVert_{L^2} + \lVert u\rVert_{\dot{H}^{m+1}}^2
\nonumber\\
&\le \lVert \Lambda^m u\rVert_{L^2} \, \lVert \nabla u\rVert_{L^\infty} \, \lVert \Lambda^m u\rVert_{L^2}
+ \lVert u\rVert_{L^\infty} \, \lVert \nabla \Lambda^m u\rVert_{L^2} \, \lVert \Lambda^m u\rVert_{L^2}+ \lVert u\rVert_{\dot{H}^{m+1}}^2
\nonumber\\
& \le \lVert \Lambda^{m} u\rVert_{L^2} \lVert \Lambda^{m} u\rVert_{L^2} \lVert \Lambda^{m}u\rVert_{L^2}+ \lVert \Lambda^{m} u \rVert_{L^2} \lVert \Lambda^{m+1}u\rVert_{L^2} \lVert \Lambda^{m} u \rVert_{L^2}+ \lVert u\rVert_{\dot{H}^{m}}\lVert u\rVert_{\dot{H}^{m+2}}
\nonumber\\
&\le C \lVert \Lambda^m u\rVert_{L^2}^{3}
+C \lVert \Lambda^m u\rVert_{L^2}^{4}+\frac{1}{4}\lVert \Lambda^{m+2} u\rVert_{L^2}^{2}+C\lVert \Lambda^{m} u\rVert_{L^2}^{2}
+ \frac{1}{4}\lVert \Lambda^{m+2}u\rVert_{L^2}^{2}
\nonumber\\
&\le C ( 1 +\lVert \Lambda^m u\rVert_{L^2}+ \lVert \Lambda^m u\rVert_{L^2}^{2}) \lVert \Lambda^m u\rVert_{L^2}^2
+ \frac{1}{2} \lVert \Lambda^{m+2} u\rVert_{L^2}^2,
\label{3.1}
\end{align}
where we applied H\"older inequality, Lemma \ref{est 500k},  Young's inequality, Sobolev embedding $\dot{H}^m(\mathbb{T}^d) \hookrightarrow L^{\infty}(\mathbb{T}^d)$ with $m>\frac{d}{2}$, and and $H^{s_2} \subset H^{s_1}$ for $s_2 > s_1$.    
From the above computations,  we obtain the following:
\begin{align}
\frac{d}{dt} \lVert \Lambda^m u\rVert_{L^2}^2
\le C ( 1+ \lVert \Lambda^m u\rVert_{L^2} + \lVert \Lambda^m u\rVert_{L^2}^{2}) \lVert \Lambda^m u\rVert_{L^2}^2.
\label{}
\end{align}
We then obtain,
\begin{align}
\lVert \Lambda^m u(t)\rVert_{L^2}^2 \;=\; Z(t)\;\le\;\frac{1+\lVert \Lambda^m u(0)\rVert_{L^2}^2}{\,1-C'(1+\lVert \Lambda^m u(0)\rVert_{L^2}^2)\,t\,}\;-\;1,\quad
0\le t<\frac{1}{C'(1+\lVert \Lambda^m u(0)\rVert_{L^2}^2)}
\label{3.3}.
\end{align}
By Grönwall's inequality, we deduce that 
\begin{align} \label{3.4}
\lVert u(t)\rVert_{\dot{H}^m}^2 \leq \lVert u_0\rVert_{\dot{H}^m}^2 \exp ( C \int_0^t ( 1+ \lVert u(\tau)\rVert_{\dot{H}^m} + \lVert u(\tau)\rVert_{\dot{H}^m}^2) d\tau).
\end{align}
In view of inequality \eqref{3.3}, the right-hand side of \eqref{3.4} is finite for all $0 \leq t < (C'(1+\lVert \Lambda^m u(0)\rVert_{L^2}^2))^{-1}$. We therefore
$$
u \in L^\infty([0, T); H^m(\mathbb{T}^d)) \cap L^2([0, T); H^{m+2}(\mathbb{T}^d)),
$$ from \eqref{3.1} with the additional $L^2$-estimate.
By using $u\in L^2([0,T);H^{m+2}])$, we obtain $u_t \in L^2(0,T;H^{m-2})$. Then, together with Lemma \ref{2.7}, we immediately obtain:
$$
u \in C([0, T]; H^m).
$$
Moreover, it is straightforward to verify that the nonlinear operator
$$
F(u) := - (u \cdot \nabla) u - \Delta u - \Delta^2 u,
$$
is locally Lipschitz from $H^m$ into $H^{m-4}$ with $L^{\infty}(\mathbb{T}^d) \hookrightarrow H^m(\mathbb{T}^d)\, \textrm{if}\; m>\frac{d}{2}$. Therefore, 
it follows from Picard-Lindel\"of Theorem, e.g., see \cite[Theorem~3.1]{MB}, in Banach Spaces that
$$
\partial_t u \in C([0, T]; H^{m-4}),$$  implying
$$
u \in C^1([0, T]; H^{m-4}).
$$
\end{proof}
Based on this result, we prove that $u$ is Sobolev regular, i.e., that $u \in H^k(D \times [0,T])$ for some $k \in \mathbb{N}$, provided that $m$ is large enough.

For the reader’s convenience, we recall the several following classical results that are used throughout the manuscript.
\subsection{A few  Classical results}
\begin{proposition}
\label{Prop 2.6}
Let $X$ Banach space and suppose $u, w \in L^1(0, T; X)$. Then the following are equivalent:
\begin{enumerate}
\item $\partial_t u = w$ in the weak sense;

\item There exists $\xi \in X$ such that, for a.e. $t \in (0, T)$,
$$
u(t) = \xi + \int_0^t w(s)\, ds.
$$
\item For every $v \in X',$ in the weak sense it holds that
$$
\frac{d}{dt}(u, v) = (w, v).
$$
\end{enumerate}
Moreover, if one (and thus all) of these conditions holds, then $u$ can be altered on a nullset of times so that it belongs to $C([0, T]; X)$.
\end{proposition}
\begin{lemma}[see \cite{61}]
\label{2.7}
Let $V, H, V'$ be three Hilbert spaces with $V'$ being the dual of $V$. If a function $u$ belongs to $L^2(0,T;V)$ and its derivative $u'$ belongs to $L^2(0,T;V')$, then $u$ is almost everywhere equal to a function continuous from $[0,T]$ into $H$ and we have the following equality, which holds in the scalar distribution sense on $(0,T)$:
\[
\frac{d}{dt} |u|^2 = 2 \langle u', u \rangle.
\]
The equality above is meaningful since the functions
\[
t \mapsto |u(t)|^2, \quad t \mapsto \langle u'(t), u(t) \rangle,
\]
are both integrable on $[0,T]$.
\end{lemma}

\begin{lemma} \label{est 500k}
For any \( f \in \dot{H}^1(\mathbb{T}^d) \cap \dot{H}^3(\mathbb{T}^d) \), the following inequality holds:
\begin{align*} \label{}
\lVert f\rVert_{\dot{H}^2(\mathbb{T}^d)}^2 \leq \left( \int_{\mathbb{T}^d} |\xi|^6 |\hat{f}(\xi)|^2 \, d\xi \right)^{1/2} \left( \int_{\mathbb{T}^d} |\xi|^2 |\hat{f}(\xi)|^2 \, d\xi \right)^{1/2}
= \|f\|_{\dot{H}^3} \|f\|_{\dot{H}^1}.
\end{align*}
\end{lemma}
\begin{lemma}[Sobolev embedding Theorem, see \cite{MB}]  
\label{2.8K}
The space \( H^{s+k}(\mathbb{T}^d) \), for \( s > N/2 \), \( k \in \mathbb{Z}_+ \cup \{0\} \) is continuously embedded in the space \( C^k(\mathbb{T}^d) \), and there exists a constant \( c > 0 \) such that  
\[ |v|_{C^k} \leq c_k \|v\|_{H^{s+k}}, \quad \text{for any } v \in H^{s+k}(\mathbb{T}^d).
\]
\end{lemma}
\begin{lemma}[see \cite{BA}]
\label{lemma_3.12k}
Let $d \in \mathbb{N}$, $1 \le s < r < \infty$, and $\beta > 0$ such that
\[
\|g\|_{L^r(\mathbb{T}^d)}
\le C\,
\|g\|_{\dot{B}^{-\beta}_{\infty,\infty}(\mathbb{T}^d)}^{\,1-\alpha}
\|g\|_{\dot{B}^{\gamma}_{s,s}(\mathbb{T}^d)}^{\,\alpha},
\]
where
\[
\alpha= \frac{s}{r}, 
\qquad 
\gamma= \beta\!\left(\frac{r}{s} - 1\right).
\]
\end{lemma}
The proof is similar to that in $\mathbb{R}^d$ using toroidal blocks; see \cite{BA} for details.
\begin{theorem}[see \cite{DR}]
\label{2.10CT}Let $d, n \ge 2, \ m \ge 3, \ \delta > 0, \ a_i, b_i \in \mathbb{Z}$ with $a_i < b_i$ for $1 \le i \le d$, 
$\Omega = \prod_{i=1}^d [a_i, b_i] $ and $f \in H^m(\Omega)$.
Then for every $N \in \mathbb{N}$ with $N > 5$, there exists a tanh neural network $\widehat{f}^N$ with two hidden layers, one of width at most
$$
3 \left\lceil \frac{m+n-2}{2} \right\rceil |P_{m-1, d+1}| + \sum_{i=1}^d (b_i - a_i)(N-1),
$$
and another of width at most
$$
3 \left\lceil \frac{d+n}{2} \right\rceil |P_{d+1, d+1}| N^d \prod_{i=1}^d (b_i - a_i),
$$
such that for $k \in \{0, 1, 2\}$ it holds that
$$
\| f - \widehat{f}^N \|_{H^k(\Omega)} \le 2^k 3^d \, C_{k,m,d,f} \,(1+\delta)\, \ln^k\!\left( \beta_{k,\delta,d,f} \, N^{d+m+2} \right) N^{-m+k},
$$
and where we define
$$
\beta_{k,\delta,d,f} =
\frac{5 \cdot 2^k \, \max\{\prod_{i=1}^d (b_i-a_i), \, d\} \, \max\{\|f\|_{W^{k,\infty}(\Omega)}, 1\}}
     {3^d \delta \min\{1, \, C_{k,m,d,f}\}},
$$
$$
C_{k,m,d,f} = \max_{0 \le \ell \le k}
\left( \frac{d+\ell-1}{\ell} \right)^{1/2}
\frac{((m-\ell)!)^{1/2}}{\left( \frac{m-\ell}{d} ! \right)^{1/2}}
\left( \frac{3\sqrt{d}}{\pi} \right)^{m-\ell} \, |f|_{H^m}.
$$
Moreover, the weights of $\widehat{f}^N$ scale as $O(N \ln(N) + N^\gamma)$ with
$\gamma = \max\{ m^2, \, d(2+m+d) \} / n$.
\end{theorem}
\begin{lemma}[Multiplicative trace inequality, see \cite{HIp}]  \label{MTI}
Let $d \ge 2$, $\Omega \subset \mathbb{R}^d$ be a Lipschitz domain, and let $\gamma_0 : H^1(\Omega) \to L^2(\partial\Omega)$ : $u \mapsto u|_{\partial\Omega}$ be the trace operator. Denote by $h_\Omega$ the diameter of $\Omega$ and by $\rho_\Omega$ the radius of the largest $d$-dimensional ball that can be inscribed into $\Omega$. Then it holds that
$$
\|\gamma_0 u\|_{L^2(\partial\Omega)} \le \sqrt{\frac{2 \max \{ 2 h_\Omega,\, d \}}{\rho_\Omega}} \, \|u\|_{H^1(\Omega)}.
$$
\end{lemma}
\begin{lemma}[see \cite{DE}] Let $d,n,L,W \in \mathbb{N}$, and let $u_\theta : \mathbb{R}^{d+1} \to \mathbb{R}^{d+1}$ be a neural network with $\theta \in \Theta_{L,W,R}$ for $L \geq 2, R, W \geq 1$ (cf. Definition \ref{4.1N}). Assume that $\|\sigma\|_{C^n} \geq 1$. Then it holds for $1 \leq j \leq d+1$ that 
\begin{align*} 
\|(u_\theta)_j\|_{C^n} \leq 16^L (d+1)^{2n} \Bigl(e^{2}n^{4} W^3 R^n \|\sigma\|_{C^n}\Bigr)^{nL}.
\end{align*}
\label{1KO}
\end{lemma}


\begin{thebibliography}{100}
\addtolength{\leftmargin}{0.2in}
\setlength{\itemindent}{-0.2in}

\bibitem{1} Ambrose, D.M., and Mazzucato, A.L. (2021). Global solutions of the two-dimensional Kuramoto-Sivashinsky equation with a linearly growing mode in each direction. \textit{J. Nonlinear Sci.} \textbf{31}, 96.

\bibitem{BA} Bahouri, H., Danchin, R., and Chemin, J.-Y. (2011). \textit{Fourier Analysis and Nonlinear Partial Differential Equations.} Springer Series of Comprehensive Studies in Mathematics, Springer-Verlag, Berlin Heidelberg. doi:10.1007/978-3-642-16830-7.

\bibitem{4} Benachour, S., Kukavica, I., Rusin, W., and Ziane, M. (2014). Anisotropic estimates for the two-dimensional Kuramoto-Sivashinsky equation. \textit{J. Dyn. Differ. Equ.} \textbf{26}, 461–476.

\bibitem{CH} Cheskidov, A., and Shvydkoy, R. (2010). The regularity of weak solutions of the 3D Navier-Stokes equations in $B^{-1}_{\infty,\infty}$. \textit{Arch. Ration. Mech. Anal.} \textbf{195}(1), 159–169.

\bibitem{13} Coti-Zelati, M., Dolce, M., Feng, Y., and Mazzucato, A.L. (2021). Global existence for the two-dimensional Kuramoto-Sivashinsky equation with a shear flow. \textit{J. Evol. Equ.} \textbf{21}, 5079–5099.

\bibitem{14p} Cuomo, S., Di Cola, V.S., Giampaolo, F., Rozza, G., Raissi, M., and Piccialli, F. (2022). Scientific machine learning through physics-informed neural networks: where we are and what’s next. \textit{arXiv preprint} arXiv:2201.05624.

\bibitem{DE} De Ryck, T., Jagtap, A.D., and Mishra, S. (2023). Error estimates for physics-informed neural networks approximating the Navier-Stokes equations. \textit{IMA J. Numer. Anal.}

\bibitem{DR} De Ryck, T., Lanthaler, S., and Mishra, S. (2021). On the approximation of functions by tanh neural networks. \textit{Neural Networks} \textbf{143}, 732–750.

\bibitem{DE1} De Ryck, T., and Mishra, S. (2021). Error analysis for physics-informed neural networks (PINNs) approximating Kolmogorov PDEs. \textit{arXiv preprint} arXiv:2106.14473.

\bibitem{15a} Dissanayake, M.W.M.G., and Phan-Thien, N. (1994). Neural-network-based approximations for solving partial differential equations. \textit{Commun. Numer. Methods Eng.} \textbf{10}(3), 195–201.

\bibitem{FA} Fang, D., and Qian, C. (2014). Regularity criterion for 3D Navier-Stokes equations in Besov spaces. \textit{Commun. Pure Appl. Anal.} \textbf{13}, 585–603.

\bibitem{17} Feng, Y., and Mazzucato, A.L. (2021). Global existence for the two-dimensional Kuramoto-Sivashinsky equation with advection. \textit{Commun. PDE} \textbf{47}, 279–306.

\bibitem{HIp} Hiptmair, R., and Schwab, C. (2008). \textit{Numerical Methods for Elliptic and Parabolic Boundary Value Problems.} ETH Z\"urich.

\bibitem{KA21} Karniadakis, G.E., Kevrekidis, I.G., Lu, L., Perdikaris, P., Wang, S., and Yang, L. (2021). Physics-informed machine learning. \textit{Nat. Rev. Phys.} \textbf{3}(6), 422–440.


\bibitem{33} Kukavica, I., and Massatt, D. (2023). On the global existence for the Kuramoto-Sivashinsky equation. \textit{J. Dyn. Differ. Equ.} \textbf{35}, 69–85.

\bibitem{34} Kuramoto, Y., and Tsuzuki, T. (1975). On the formation of dissipative structures in reaction–diffusion systems. \textit{Prog. Theor. Phys.} \textbf{54}, 687–699.

\bibitem{35} Kuramoto, Y., and Tsuzuki, T. (1976). Persistent propagation of concentration waves in dissipative media far from equilibrium. \textit{Prog. Theor. Phys.} \textbf{55}, 365–369.

\bibitem{37} Larios, A., Rahman, M.M., and Yamazaki, K. (2022). Regularity criteria for the Kuramoto-Sivashinsky equation in dimensions two and three. \textit{J. Nonlinear Sci.} \textbf{32}, 1–33.

\bibitem{38} Larios, A., and Titi, E.S. (2016). Global regularity versus finite-time singularities: some paradigms on the effect of boundary conditions and certain perturbations. In: \textit{Recent Progress in the Theory of the Euler and Navier-Stokes Equations}, vol. 430, Cambridge University Press, pp. 96–125.

\bibitem{39} Larios, A., and Yamazaki, K. (2020). On the well-posedness of an anisotropically reduced two-dimensional Kuramoto-Sivashinsky equation. \textit{Physica D} \textbf{411}, 132560.

\bibitem{40l} Lagaris, I.E., Likas, A., and Fotiadis, D.I. (1998). Artificial neural networks for solving ordinary and partial differential equations. \textit{IEEE Trans. Neural Netw.} \textbf{9}(5), 987–1000.

\bibitem{41p} Lagaris, I.E., Likas, A.C., and Papageorgiou, D.G. (2000). Neural-network methods for boundary value problems with irregular boundaries. \textit{IEEE Trans. Neural Netw.} \textbf{11}(5), 1041–1049.

\bibitem{MB} Majda, A., and Bertozzi, A. (2002). \textit{Vorticity and Incompressible Flow.} Cambridge Texts in Applied Mathematics, vol. 27. Cambridge University Press, Cambridge.

\bibitem{41} Massatt, D. (2022). On the well-posedness of the anisotropically reduced two-dimensional Kuramoto-Sivashinsky equation. \textit{Discrete Contin. Dyn. Syst. B} \textbf{27}, 6023.

\bibitem{42} Michelson, D.M., and Sivashinsky, G.I. (1977). Nonlinear analysis of hydrodynamic instability in laminar flames–II. Numerical experiments. \textit{Acta Astronaut.} \textbf{4}, 1207–1221.

\bibitem{MS} Mishra, S., and Molinaro, R. (2022). Estimates on the generalization error of physics-informed neural networks for approximating a class of inverse problems for PDEs. \textit{IMA J. Numer. Anal.} \textbf{42}(2), 981–1022.

\bibitem{MS1} Mishra, S., and Molinaro, R. (2022). Estimates on the generalization error of physics-informed neural networks for approximating PDEs. \textit{IMA J. Numer. Anal.} \textbf{43}(1), 1–43.

\bibitem{46} Nicolaenko, B., and Scheurer, B. (1984). Remarks on the Kuramoto-Sivashinsky equation. \textit{Physica D} \textbf{12}, 391–395.

\bibitem{50} Pokhozhaev, S.I. (2008). On the blow-up of solutions of the Kuramoto-Sivashinsky equation. \textit{Math. Sb.} \textbf{199}, 97–106.

\bibitem{RAM} Raissi, M., and Karniadakis, G.E. (2018). Hidden physics models: machine learning of nonlinear partial differential equations. \textit{J. Comput. Phys.} \textbf{357}, 125–141.

\bibitem{RA} Raissi, M., Perdikaris, P., and Karniadakis, G.E. (2019). Physics-informed neural networks: a deep learning framework for solving forward and inverse problems involving nonlinear partial differential equations. \textit{J. Comput. Phys.} \textbf{378}, 686–707.

\bibitem{54} Sell, G.R., and Taboada, M. (1992). Local dissipativity and attractors for the Kuramoto-Sivashinsky equation in thin 2D domains. \textit{Nonlinear Anal.} \textbf{18}, 671–687.

\bibitem{55} Sivashinsky, G.I. (1977). Nonlinear analysis of hydrodynamic instability in laminar flames. I. Derivation of basic equations. \textit{Acta Astronaut.} \textbf{4}, 1177–1206.

\bibitem{56} Sivashinsky, G.I. (1980). On flame propagation under conditions of stoichiometry. \textit{SIAM J. Appl. Math.} \textbf{39}, 67–82.

\bibitem{59} Tadmor, E. (1986). The well-posedness of the Kuramoto-Sivashinsky equation. \textit{SIAM J. Math. Anal.} \textbf{17}, 884–893.

\bibitem{61} Temam, R. (2001). \textit{Navier-Stokes Equations: Theory and Numerical Analysis.} AMS Chelsea Publishing, reprint of the 1984 edition.

\bibitem{THS} Tsurumi, H. (2019). Well-posedness and ill-posedness of the stationary Navier-Stokes equations in toroidal Besov spaces. \textit{Nonlinearity} \textbf{32}(10).

\bibitem{JU} Wu, J. (2004). The generalized incompressible Navier-Stokes equations in Besov spaces. \textit{Dyn. Partial Differ. Eq.} \textbf{1}, 381–400.

\end{thebibliography}
\end{document}